\documentclass[final,3p]{elsarticle}
 \usepackage{graphics}
 \usepackage{graphicx}
 \usepackage{epsfig}
\usepackage{amssymb}
 \usepackage{amsthm}
 \usepackage{lineno}
 \usepackage{amsmath}
   \numberwithin{equation}{section}
\usepackage{mathrsfs}

\newtheorem{thm}{Theorem}[section]
\newtheorem{cor}[thm]{Corollary}
\newtheorem{lem}[thm]{Lemma}
\newtheorem{prop}[thm]{Proposition}
\newtheorem{defn}[thm]{Definition}

\biboptions{sort&compress,square}
\journal{~}
\begin{document}
\begin{frontmatter}
\author{Sining Wei}
\ead{weisn835@nenu.edu.cn}
\author{Yong Wang\corref{cor2}}
\ead{wangy581@nenu.edu.cn}
\cortext[cor2]{Corresponding author.}

\address{School of Mathematics and Statistics, Northeast Normal University,
Changchun, 130024, P.R.China}

\title{Statistical de Rham Hodge operators and the Kastler-Kalau-Walze type theorem for manifolds with boundary }
\begin{abstract}
In this paper, we give the Lichnerowicz type formulas for statistical de Rham Hodge operators. Moreover, Kastler-Kalau-Walze type theorems for statistical de Rham Hodge operators on compact manifolds with (respectively without) boundary are proved.
\end{abstract}
\begin{keyword} Statistical de Rham Hodge operators; Lichnerowicz type formulas; Kastler-Kalau-Walze type theorem; Noncommutative residue.\\

\end{keyword}
\end{frontmatter}
\section{Introduction}
The noncommutative residue which is found in \cite{Gu,Wo} plays a prominent role in noncommutative geometry. For this reason, it has been studied extensively by geometers. Connes derived a conformal 4-dimensional Polyakov action analogy by the noncommutative residue in \cite{Co1}. Connes gived the relation of the noncommutative residue on a compact manifold $M$ and the Dixmier's trace on pseudodifferential operators of order $-{\rm {dim}}M$ in \cite{Co2}. On the other hand,
Connes has also observed that the noncommutative residue of the square of the inverse of the Dirac
operator was proportional to the Einstein-Hilbert action, which is called the Kastler-Kalau-Walze type theorem now. Kastler gave a
brute-force proof of this theorem in \cite{Ka}. Kalau and Walze \cite{KW} proved this theorem in the normal coordinates system simultaneously.
Ackermann proved that
the Wodzicki residue  of the square of the inverse of the Dirac operator ${\rm  Wres}(D^{-2})$ in turn is essentially the second coefficient
of the heat kernel expansion of $D^{2}$ in \cite{Ac}.

On the other hand, Wang generalized the Connes' results to the case of manifolds with boundary in \cite{Wa1,Wa2},
and proved the Kastler-Kalau-Walze type theorem for the Dirac operator and the signature operator on lower-dimensional manifolds
with boundary \cite{Wa3}. In \cite{Wa3,Wa4}, Wang computed $\widetilde{{\rm Wres}}[\pi^+D^{-1}\circ\pi^+D^{-1}]$ and $\widetilde{{\rm Wres}}[\pi^+D^{-2}\circ\pi^+D^{-2}]$ about symmetric operators, under the circumstances the boundary term vanished. Moreover, Wang got a nonvanishing boundary term for $\widetilde{{\rm Wres}}[\pi^+D^{-1}\circ\pi^+D^{-3}]$ in \cite{Wa5} and give a theoretical explanation for gravitational action on boundary. In others words, Wang provides a kind of method to study the Kastler-Kalau-Walze type theorem for manifolds with boundary.

In \cite{LI}, Iochum and Levy computed heat kernel coefficients for Dirac operators with one-form perturbations and proved that there are no tadpoles for compact spin manifolds without boundary. Recently, we have studied the Lichnerowicz-type formulas for modified Novikov operators. We prove Kastler-Kalau-Walze-type theorems for modified Novikov operators on compact manifolds with (respectively without) a boundary in \cite{We}.
In \cite{Op}, Barbara Opozda introduces statistical de Rham Hodge operators on manifolds with statistical structure.
The aim of this paper is
to prove the Kastler-Kalau-Walze type theorem for statistical de Rham Hodge operators on manifolds without~(with) boundary, and also give the Lichnerowicz formulas about statistical de Rham Hodge operators.

The paper is organized as follows: in Section 2, we give the definition of statistical de Rham Hodge operators and its Lichnerowicz formulas. We also give the basic facts and formulas about the noncommutative residue for manifolds with boundary. In Section 3 and in Section 4, we give some  expressions and symbols of operators associate with statistical de Rham Hodge operators. Moveover, we also prove the Kastler-Kalau-Walze type theorem for statistical de Rham Hodge operators on 4-dimensional and 6-dimensional manifolds with boundary.

\section{Statistical de Rham Hodge Operators and its Lichnerowicz formula}

In this section, we prepare some basic notions about
statistical de Rham Hodge operators.
For details of the geometry of statistical manifolds, see \cite{Op}.

Let $M$ be a $n$-dimensional ($n\geq 3$)
Riemannian manifold with a positive definite Riemannian tensor field $g$,
$s$ be a tensor field and $\hat{\nabla}$ be a connection.
We assume that $M$ is oriented.
Let $\nabla^L$ be the Levi-Civita connection for $g$ and
$Vol_M$ be the volume form determined by $g$.

For all $X, Y, Z\in T_xM, x\in M$, if $\hat{\nabla}$ is
satisfying the following Codazzi condition:
\begin{eqnarray}
(\hat{\nabla}_X g)(Y, Z) = (\hat{\nabla}_Y g)(X, Z),
\end{eqnarray}
we call a structure $(g, \hat{\nabla})$ is a statistical structure, and
a connection $\hat{\nabla}$ is a statistical connection for $g$.
Moreover, we define a statistical manifold with statistical structure by $(M, g, \hat{\nabla})$.

We suppose that $(M, g, \hat{\nabla})$ is statistical manifold. Let $K$ be tensor field, where $K:\Gamma(TM)\times\Gamma(TM)\rightarrow\Gamma(TM)$ and $\Gamma(TM)$ is the vector field of $M$. If $K$ is the difference tensor between $\hat{\nabla}$ and $\nabla^L$, that is
\begin{eqnarray}
\hat{\nabla}_XY = \nabla^L_XY+K_XY,
\end{eqnarray}

$K(X,Y)$ will stand for $K_XY$.
The de Rham derivative $d$ is an elliptic differential operator on $C^\infty(M;\wedge^*T^*M)$.
Then we have the de Rham coderivative $\delta=d^*$ and the symmetric operators $D=d+\delta$.
The standard Hodge Laplacian is defined by
\begin{eqnarray}
\Delta = \delta d+d\delta.
\end{eqnarray}

For a statistical manifold $(M, g, \hat{\nabla})$, we will investigate a (Lichnerowicz) Laplacian relative to the connection $\hat{\nabla}$. If $f$ is a function, then we set
\begin{eqnarray}
\Delta^{\hat{\nabla}} f=-div^{\hat{\nabla}} grad f.
\end{eqnarray}
We now extend the definition (2.4) and set $E=trace_gK(\cdot,\cdot)$. For any differential form $\nu$, we set
\begin{eqnarray}
\Delta^{\hat{\nabla}} \nu=(\delta-l(E))d\nu+d(\delta-l(E))\nu,
\end{eqnarray}
where $l(E)$ is the contraction operator.

By the above definition and Lemma 6.4 in \cite{Op}, we have
\begin{eqnarray}
\Delta^{\hat{\nabla}}=(\delta-l(E)+d)^2.
\end{eqnarray}



For $v\in \Gamma(TM)$, we define the generalized statistical de Rham Hodge operators by
\begin{eqnarray}
D_i&=&d+\delta+\lambda_il(v),~(i=1,2);\nonumber\\
~D^*_i&=&d+\delta+\lambda_i\varepsilon(v^*),~(i=1,2),
\end{eqnarray}
where $\lambda_i$ is a real number.


In the local coordinates $\{x_i; 1\leq i\leq n\}$ and the
fixed orthonormal frame $\{e_1,\cdots,e_n\}$, the connection matrix $(\omega_{s,t})$ is defined by
\begin{equation}
\nabla^L(e_1,\cdots,e_n)= (e_1,\cdots,e_n)(\omega_{s,t}).
\end{equation}

Let $\varepsilon (e_j* ),~l (e_j)$ be the exterior and interior multiplications respectively and $c(e_j)$ be the Clifford action. Write
\begin{equation}
c(e_j)=\varepsilon (e_j*)-l(e_j);~~
\bar{c}(e_j)=\varepsilon(e_j*)+l(e_j).
\end{equation}

Moreover, we assume that $\partial_{i}$ is a natural local frame on $TM$
and $(g^{ij})_{1\leq i,j\leq n}$ is the inverse matrix associated with the metric
matrix  $(g_{ij})_{1\leq i,j\leq n}$ on $M$.
The statistical de Rham Hodge operators $D_i$ and $D^*_i$~$(i=1,2)$ are defined by
\begin{equation}
D_i=d+\delta+\lambda_il(v)=\sum^n_{i=1}c(e_i)\bigg[e_i+\frac{1}{4}\sum_{s,t}\omega_{s,t}
(e_i)[\bar{c}(e_s)\bar{c}(e_t)
-c(e_s)c(e_t)]\bigg]+\lambda_il(v);
\end{equation}
\begin{equation}
D^*_i=d+\delta+\lambda_i\varepsilon(v^*)=\sum^n_{i=1}c(e_i)\bigg[e_i+\frac{1}{4}\sum_{s,t}\omega_{s,t}
(e_i)[\bar{c}(e_s)\bar{c}(e_t)
-c(e_s)c(e_t)]\bigg]+\lambda_i\varepsilon(v^*).
\end{equation}

Let $g^{ij}=g(dx_{i},dx_{j})$, $\xi=\sum_{k}\xi_{j}dx_{j}$ and $\nabla^L_{\partial_{i}}\partial_{j}=\sum_{k}\Gamma_{ij}^{k}\partial_{k}$,  we denote that
\begin{eqnarray}
&&\sigma_{i}=-\frac{1}{4}\sum_{s,t}\omega_{s,t}
(e_i)c(e_s)c(e_t)
;~~~a_{i}=\frac{1}{4}\sum_{s,t}\omega_{s,t}
(e_i)\bar{c}(e_s)\bar{c}(e_t);\nonumber\\
&&\xi^{j}=g^{ij}\xi_{i};~~~~\Gamma^{k}=g^{ij}\Gamma_{ij}^{k};~~~~\sigma^{j}=g^{ij}\sigma_{i};
~~~~a^{j}=g^{ij}a_{i}.
\end{eqnarray}
Then the statistical de Rham Hodge operators $D_i$ and $D^*_i$ can be written as
\begin{equation}
D_i=\sum^n_{i=1}c(e_i)[e_i+a_{i}+\sigma_{i}]
+\lambda_il(v),~(i=1,2);
\end{equation}
\begin{equation}
D_i^*=\sum^n_{i=1}c(e_i)[e_i+a_{i}+\sigma_{i}]
+\lambda_i\varepsilon(v^*),~(i=1,2).
\end{equation}

On the other hand, we recall some basic facts and formulas about Boutet de
Monvel's calculus and the definition of the noncommutative residue for manifolds with boundary (see details in Section 2 in \cite{Wa3}).

Let $U\subset M$ be a collar neighborhood of $\partial M$ which is diffeomorphic with $\partial M\times [0,1)$. By the definition of $h(x_n)\in C^{\infty}([0,1))$
and $h(x_n)>0$, there exists $\tilde{h}\in C^{\infty}((-\varepsilon,1))$ such that $\tilde{h}|_{[0,1)}=h$ and $\tilde{h}>0$ for some
sufficiently small $\varepsilon>0$. Then there exists a metric $\widehat{g}$ on $\widehat{M}=M\bigcup_{\partial M}\partial M\times
(-\varepsilon,0]$ which has the form on $U\bigcup_{\partial M}\partial M\times (-\varepsilon,0 ]$
\begin{equation}
\widehat{g}=\frac{1}{\tilde{h}(x_{n})}g^{\partial M}+dx _{n}^{2} ,
\end{equation}
such that $\widehat{g}|_{M}=g$. We fix a metric $\widehat{g}$ on the $\widehat{M}$ such that $\widehat{g}|_{M}=g$.

Let  \begin{equation*}
F:L^2({\bf R}_t)\rightarrow L^2({\bf R}_\lambda);~F(u)(\lambda)=\int e^{-ivt}u(t)dt
\end{equation*}
denote the Fourier transformation and
$\varphi(\overline{{\bf R}^+}) =r^+\varphi({\bf R})$ (similarly define $\varphi(\overline{{\bf R}^-}$)), where $\varphi({\bf R})$
denotes the Schwartz space and
\begin{equation*}
r^{+}:C^\infty ({\bf R})\rightarrow C^\infty (\overline{{\bf R}^+});~ f\rightarrow f|\overline{{\bf R}^+};~
\overline{{\bf R}^+}=\{x\geq0;x\in {\bf R}\}.
\end{equation*}
We define $H^+=F(\varphi(\overline{{\bf R}^+}));~ H^-_0=F(\varphi(\overline{{\bf R}^-}))$ which are orthogonal to each other. We have the following
 property: $h\in H^+~(H^-_0)$ if and only if $h\in C^\infty({\bf R})$ which has an analytic extension to the lower (upper) complex
half-plane $\{{\rm Im}\xi<0\}~(\{{\rm Im}\xi>0\})$ such that for all nonnegative integer $l$,
 \begin{equation*}
\frac{d^{l}h}{d\xi^l}(\xi)\sim\sum^{\infty}_{k=1}\frac{d^l}{d\xi^l}(\frac{c_k}{\xi^k}),
\end{equation*}
as $|\xi|\rightarrow +\infty,{\rm Im}\xi\leq0~({\rm Im}\xi\geq0)$.

 Let $H'$ be the space of all polynomials and $H^-=H^-_0\bigoplus H';~H=H^+\bigoplus H^-.$ Denote by $\pi^+~(\pi^-)$ respectively the
 projection on $H^+~(H^-)$. For calculations, we take $H=\widetilde H=\{$rational functions having no poles on the real axis$\}$ ($\tilde{H}$
 is a dense set in the topology of $H$). Then on $\tilde{H}$,
 \begin{equation}
\pi^+h(\xi_0)=\frac{1}{2\pi i}\lim_{u\rightarrow 0^{-}}\int_{\Gamma^+}\frac{h(\xi)}{\xi_0+iu-\xi}d\xi,
\end{equation}
where $\Gamma^+$ is a Jordan close curve
included ${\rm Im}(\xi)>0$ surrounding all the singularities of $h$ in the upper half-plane and
$\xi_0\in {\bf R}$. Similarly, define $\pi'$ on $\tilde{H}$,
\begin{equation}
\pi'h=\frac{1}{2\pi}\int_{\Gamma^+}h(\xi)d\xi.
\end{equation}
So, $\pi'(H^-)=0$. For $h\in H\bigcap L^1({\bf R})$, $\pi'h=\frac{1}{2\pi}\int_{{\bf R}}h(v)dv$ and for $h\in H^+\bigcap L^1({\bf R})$, $\pi'h=0$.

Next we give the basic notions of Laplace type operators. Let $M$ be a smooth compact oriented Riemannian $n$-dimensional manifolds without boundary and $V'$ be a vector bundle on $M$. Any differential operator $P$ of Laplace type has locally the form
\begin{equation}
P=-(g^{ij}\partial_i\partial_j+A^i\partial_i+B),
\end{equation}
where
$A^{i}$ and $B$ are smooth sections
of the endomorphism $\textrm{End}(V')$ on $M$. If $P$ is a Laplace type
operator with the form (2.15), then there is a unique
connection $\nabla$ on $V'$ and a unique endomorphism $E'$ such that
 \begin{equation}
P=-\Big[g^{ij}(\nabla_{\partial_{i}}\nabla_{\partial_{j}}-
 \nabla_{\nabla^{L}_{\partial_{i}}\partial_{j}})+E'\Big],
\end{equation}
where $\nabla^{L}$ is the Levi-Civita connection on $M$. Moreover
(with local frames of $T^{*}M$ and $V'$), $\nabla_{\partial_{i}}=\partial_{i}+\omega_{i} $
and $E'$ are related to $g^{ij}$, $A^{i}$ and $B$ through
 \begin{eqnarray}
&&\omega_{i}=\frac{1}{2}g_{ij}\big(A^{i}+g^{kl}\Gamma_{ kl}^{j} \texttt{id}\big),\\
&&E'=B-g^{ij}\big(\partial_{i}(\omega_{j})+\omega_{i}\omega_{j}-\omega_{k}\Gamma_{ ij}^{k} \big),
\end{eqnarray}
where $\Gamma_{ kl}^{j}$ is the  Christoffel coefficient of $\nabla^{L}$.

By the above definitions, we establish the main theorem in this section. One has the following Lichneriowicz formulas.

\begin{thm} The following equalities hold:
\begin{eqnarray}
D_1D_2&=&-\Big[g^{ij}(\nabla_{\partial_{i}}\nabla_{\partial_{j}}-
\nabla_{\nabla^{L}_{\partial_{i}}\partial_{j}})\Big]
-\frac{1}{8}\sum_{ijkl}R_{ijkl}\bar{c}(e_i)\bar{c}(e_j)
c(e_k)c(e_l)+\frac{1}{4}s
\nonumber\\
&&
+\frac{1}{4}\sum_{i}
[\lambda_2c(e_{i})l(v)+\lambda_1l(v)c(e_{i})]^2-\frac{1}{2}[\lambda_1\nabla^{TM}_{e_{j}}(l(v))c(e_{j})
-\lambda_2c(e_{j})\nabla^{TM}_{e_{j}}(l(v))],\\
D^*_2D^*_1&=&-\Big[g^{ij}(\nabla_{\partial_{i}}\nabla_{\partial_{j}}-
\nabla_{\nabla^{L}_{\partial_{i}}\partial_{j}})\Big]
-\frac{1}{8}\sum_{ijkl}R_{ijkl}\bar{c}(e_i)\bar{c}(e_j)
c(e_k)c(e_l)+\frac{1}{4}s+\frac{1}{4}\sum_{i}
[\lambda_1c(e_{i})\varepsilon(v^*)
\nonumber\\
&&
+\lambda_2\varepsilon(v^*)c(e_{i})]^2
-\frac{1}{2}[\lambda_2\nabla^{TM}_{e_{j}}(\varepsilon(v^*))c(e_{j})
-\lambda_1c(e_{j})\nabla^{TM}_{e_{j}}(\varepsilon(v^*))],\\
D_2^*D_1&=&-\Big[g^{ij}(\nabla_{\partial_{i}}\nabla_{\partial_{j}}-
\nabla_{\nabla^{L}_{\partial_{i}}\partial_{j}})\Big]
-\frac{1}{8}\sum_{ijkl}R_{ijkl}\bar{c}(e_i)\bar{c}(e_j)
c(e_k)c(e_l)
+\frac{1}{4}s+\lambda_1\lambda_2\varepsilon(v^*)l(v)\nonumber\\
&&+\frac{1}{4}\sum_{i}
[\lambda_1c(e_{i})l(v)+\lambda_2\varepsilon(v^*)c(e_{i})]^2
-\frac{1}{2}[\lambda_2\nabla^{TM}_{e_{j}}(\varepsilon(v^*))c(e_{j})
-\lambda_1c(e_{j})\nabla^{TM}_{e_{j}}(l(v))],
\end{eqnarray}
where $s$ is the scalar curvature.
\end{thm}
\begin{proof}By (2.13),
we note that
\begin{eqnarray}
D_1D_2&=&(d+\delta)^2+\lambda_2(d+\delta)l(v)
+\lambda_1l(v)(d+\delta)+\lambda_1\lambda_2[l(v)]^2.
\end{eqnarray}
By \cite{Y}, the local expression of $(d+\delta)^{2}$ is
\begin{equation}
(d+\delta)^{2}
=-\triangle_{0}-\frac{1}{8}\sum_{ijkl}R_{ijkl}\bar{c}(\widetilde{e_i})\bar{c}(\widetilde{e_j})c(\widetilde{e_k})c(\widetilde{e_l})+\frac{1}{4}s.
\end{equation}

By \cite{Y} and \cite{Ac}, we have
\begin{equation}
-\triangle_{0}=\Delta=-g^{ij}(\nabla^L_{i}\nabla^L_{j}-\Gamma_{ij}^{k}\nabla^L_{k}).
\end{equation}

\begin{eqnarray}
\lambda_2(d+\delta)l(v)+\lambda_1l(v)(d+\delta)&=&\sum_{i,j}g^{i,j}\Big[\lambda_2c(\partial_{i})l(v)
+\lambda_1l(v)c(\partial_{i})\Big]\partial_{j}
+\sum_{i,j}g^{i,j}\Big[\lambda_1l(v)c(\partial_{i})(\sigma_{i}+a_i)\nonumber\\
&&+\lambda_2c(\partial_{i})\partial_{j}(l(v))+\lambda_2c(\partial_{i})(\sigma_{i}+a_i)l(v)\Big];
\end{eqnarray}
\begin{eqnarray}
[l(v)]^2=0.
\end{eqnarray}
then we obtain
\begin{eqnarray}
D_1D_2&=&-\sum_{i,j}g^{i,j}\Big[\partial_{i}\partial_{j}
+2\sigma_{i}\partial_{j}+2a_{i}\partial_{j}-\Gamma_{i,j}^{k}\partial_{k}
+(\partial_{i}\sigma_{j})
+(\partial_{i}a_{j})
+\sigma_{i}\sigma_{j}+\sigma_{i}a_{j}+a_{i}\sigma_{j}+a_{i}a_{j} -\Gamma_{i,j}^{k}\sigma_{k}\nonumber\\
&&-\Gamma_{i,j}^{k}a_{k}\Big]
+\sum_{i,j}g^{i,j}\Big[\lambda_2c(\partial_{i})l(v)+\lambda_1l(v)c(\partial_{i})\Big]\partial_{j}
-\frac{1}{8}\sum_{ijkl}R_{ijkl}\bar{c}(e_i)\bar{c}(e_j)
c(e_k)c(e_l)+\frac{1}{4}s\nonumber\\
&&+\sum_{i,j}g^{i,j}\Big[\lambda_2c(\partial_{i})\partial_{j}(l(v))
+\lambda_2c(\partial_{i})(\sigma_{i}+a_i)l(v)+
\lambda_1l(v)c(\partial_{i})(\sigma_{i}+a_{i})\Big].
\end{eqnarray}
Similarly, we have
\begin{eqnarray}
D^*_2D^*_1&=&-\sum_{i,j}g^{i,j}\Big[\partial_{i}\partial_{j}
+2\sigma_{i}\partial_{j}+2a_{i}\partial_{j}-\Gamma_{i,j}^{k}\partial_{k}
+(\partial_{i}\sigma_{j})
+(\partial_{i}a_{j})
+\sigma_{i}\sigma_{j}+\sigma_{i}a_{j}+a_{i}\sigma_{j}+a_{i}a_{j}-\Gamma_{i,j}^{k}\sigma_{k}\nonumber\\
&&-\Gamma_{i,j}^{k}a_{k}\Big]
+\sum_{i,j}g^{i,j}\Big[\lambda_1c(\partial_{i})\varepsilon(v^*)
+\lambda_2\varepsilon(v^*)c(\partial_{i})\Big]\partial_{j}
-\frac{1}{8}\sum_{ijkl}R_{ijkl}\bar{c}(e_i)\bar{c}(e_j)
c(e_k)c(e_l)+\frac{1}{4}s\nonumber\\
&&+\sum_{i,j}g^{i,j}\Big[\lambda_1c(\partial_{i})\partial_{j}(\varepsilon(v^*))
+\lambda_1c(\partial_{i})(\sigma_{i}+a_i)\varepsilon(v^*)+
\lambda_2\varepsilon(v^*)c(\partial_{i})(\sigma_{i}+a_{i})\Big].
\end{eqnarray}
and
\begin{eqnarray}
D^*_2D_1&=&-\sum_{i,j}g^{i,j}\Big[\partial_{i}\partial_{j}
+2\sigma_{i}\partial_{j}+2a_{i}\partial_{j}-\Gamma_{i,j}^{k}\partial_{k}
+(\partial_{i}\sigma_{j})
+(\partial_{i}a_{j})
+\sigma_{i}\sigma_{j}+\sigma_{i}a_{j}+a_{i}\sigma_{j}+a_{i}a_{j} -\Gamma_{i,j}^{k}\sigma_{k}\nonumber\\
&&-\Gamma_{i,j}^{k}a_{k}\Big]
+\sum_{i,j}g^{i,j}\Big[\lambda_1c(\partial_{i})l(v)+\lambda_2\varepsilon(v^*)c(\partial_{i})\Big]\partial_{j}
-\frac{1}{8}\sum_{ijkl}R_{ijkl}\bar{c}(e_i)\bar{c}(e_j)
c(e_k)c(e_l)+\frac{1}{4}s\nonumber\\
&&+\sum_{i,j}g^{i,j}\Big[\lambda_1c(\partial_{i})\partial_{j}(l(v))
+\lambda_1c(\partial_{i})(\sigma_{i}+a_i)l(v)+
\lambda_2\varepsilon(v^*)c(\partial_{i})(\sigma_{i}+a_{i})\Big]+\lambda_1\lambda_2\varepsilon(v^*)l(v).
\end{eqnarray}

By (2.18), (2.20) and (2.30) we have
\begin{eqnarray}
(\omega_{i})_{D_1D_2}&=&\sigma_{i}+a_{i}
-\frac{1}{2}\Big[\lambda_2c(\partial_{i})l(v)+\lambda_1l(v)c(\partial_{i})\Big],\\
E'_{D_1D_2}&=&\sum_{i,j}g^{i,j}\Big[
\partial_{i}(\sigma_{j}+a_j)+\sigma_{i}\sigma_{j}+\sigma_{i}a_{j}+a_{i}\sigma_{j}
-\Gamma_{ij}^{k}\sigma_{k}-\Gamma_{ij}^{k}a_{k}+a_{i}a_{j}
-\lambda_2c(\partial_{i})\partial_{j}(l(v))-c(\partial_{i})(\sigma_{j}\nonumber\\
&&+a_j)\lambda_2l(v)
-\lambda_1l(v)c(\partial_{i})(\sigma_{j}+a_j)\Big]
+\frac{1}{8}\sum_{ijkl}R_{ijkl}\bar{c}(e_i)\bar{c}(e_j)
c(e_k)c(e_l)-\frac{1}{4}s-\lambda_1\lambda_2[l(v)]^{2}
\nonumber\\
&&-\sum_{i,j}g^{i,j}\{\partial_{i}(\sigma_{j}+a_{j})
-\frac{1}{2}\partial_{i}[c(\partial_{j})\lambda_2l(v)+\lambda_1l(v)c(\partial_{j})]
+\Big[\sigma_{i}+a_{i}-\frac{1}{2}[\lambda_2c(\partial_{i})l(v)+\lambda_1l(v)c(\partial_{i})]\Big]
\nonumber\\
&&\times\Big[\sigma_{j}+a_{j}-\frac{1}{2}[\lambda_2c(\partial_{j})l(v)+\lambda_1l(v)c(\partial_{j})]\Big]
-\Big[\sigma_{k}+a_{k}-\frac{1}{2}[c(\partial_{k})\lambda_2l(v)+\lambda_1l(v)c(\partial_{k})]\Big]\Gamma_{ ij}^{k}
\}.
\end{eqnarray}
Let $c(Y)$ denote the Clifford action, where $Y$ is a smooth vector field on $M$. Since $E'$ is globally defined on $M$, taking normal coordinates at $x_0$, we have
$\sigma^{i}(x_0)=0$, $a^{i}(x_0)=0$, $\partial^{j}[c(\partial_{j})](x_0)=0$,
$\Gamma^k(x_0)=0$, $g^{ij}(x_0)=\delta^j_i$. By (2.15), then we have
\begin{eqnarray}
E'_{D_1D_2}(x_0)&=&\frac{1}{8}\sum_{ijkl}R_{ijkl}\bar{c}(e_i)\bar{c}(e_j)
c(e_k)c(e_l)
-\frac{1}{4}s-\frac{1}{4}\sum_{i}[\lambda_2c(e_{i})l(v)+\lambda_1l(v)c(e_{i})]^2\nonumber\\
&&+\frac{1}{2}[\lambda_1\nabla^{TM}_{e_{j}}(l(v))c(e_{j})-\lambda_2c(e_{j})\nabla^{TM}_{e_{j}}(l(v))].
\end{eqnarray}
Similarly, we have
\begin{eqnarray}
E'_{D^*_2D^*_1}(x_0)&=&\frac{1}{8}\sum_{ijkl}R_{ijkl}\bar{c}(e_i)\bar{c}(e_j)
c(e_k)c(e_l)
-\frac{1}{4}s-\frac{1}{4}\sum_{i}[\lambda_1c(e_{i})\varepsilon(v^*)
+\lambda_2\varepsilon(v^*)c(e_{i})]^2\nonumber\\
&&+\frac{1}{2}[\lambda_2\nabla^{TM}_{e_{j}}(\varepsilon(v^*))c(e_{j})
-\lambda_1c(e_{j})\nabla^{TM}_{e_{j}}(\varepsilon(v^*))].
\end{eqnarray}
\begin{eqnarray}
E'_{D^*_2D_1}(x_0)&=&\frac{1}{8}\sum_{ijkl}R_{ijkl}\bar{c}(e_i)\bar{c}(e_j)
c(e_k)c(e_l)
-\frac{1}{4}s-\frac{1}{4}\sum_{i}[\lambda_1c(e_{i})l(v)+\lambda_2\varepsilon(v^*)c(e_{i})]^2\nonumber\\
&&+\frac{1}{2}[\lambda_2\nabla^{TM}_{e_{j}}(\varepsilon(v^*))c(e_{j})
-\lambda_1c(e_{j})\nabla^{TM}_{e_{j}}(l(v))]
-\lambda_1\lambda_2\varepsilon(v^*)l(v).
\end{eqnarray}
which, together with (2.19), we complete the proof.
\end{proof}

The non-commutative residue of a generalized laplacian $\Delta$ is expressed as by \cite{Ac}

\begin{equation}
(n-2)\phi(\Delta)=(4\pi)^{-\frac{n}{2}}\Gamma(\frac{n}{2})\widetilde{res}(\Delta^{-\frac{n}{2}+1}),
\end{equation}
where $\phi(\Delta)$ denotes the integral over the diagonal part of the second
coefficient of the heat kernel expansion of $\Delta$.
Since $D_1D_2$, $D^*_2D^*_1$ and $D^*_2D_1$ are generalized laplacian opeartors, we have
\begin{eqnarray}
{\rm Wres}(D_1D_2)^{-\frac{n-2}{2}}
=\frac{(n-2)(4\pi)^{\frac{n}{2}}}{(\frac{n}{2}-1)!}\int_{M}{\rm trace}(\frac{1}{6}s+E'_{D_1D_2})d{\rm Vol_{M}},
\end{eqnarray}
where ${\rm Wres}$ is the noncommutative residue.

Similarly, we have
\begin{eqnarray}
{\rm Wres}(D^*_2D^*_1)^{-\frac{n-2}{2}}
=\frac{(n-2)(4\pi)^{\frac{n}{2}}}{(\frac{n}{2}-1)!}\int_{M}{\rm trace}(\frac{1}{6}s+E'_{D^*_2D^*_1})d{\rm Vol_{M}},
\end{eqnarray}
\begin{eqnarray}
{\rm Wres}(D^*_2D_1)^{-\frac{n-2}{2}}
=\frac{(n-2)(4\pi)^{\frac{n}{2}}}{(\frac{n}{2}-1)!}\int_{M}{\rm trace}(\frac{1}{6}s+E'_{D^*_2D_1})d{\rm Vol_{M}},
\end{eqnarray}
By the Theorem 2.1 and its proof, we have
\begin{thm} For even $n$-dimensional compact oriented manifolds without boundary, the following equalities holds:
\begin{eqnarray}
{\rm Wres}(D_1D_2)^{-\frac{n-2}{2}}
&=&\frac{(n-2)(4\pi)^{\frac{n}{2}}}{(\frac{n}{2}-1)!}\int_{M}
2^n\bigg(-\frac{1}{12}s-\frac{1}{4}(\lambda^2_1+\lambda^2_2)|v|^2\bigg)d{\rm Vol_{M}}\\
{\rm Wres}(D^*_2D^*_1)^{-\frac{n-2}{2}}
&=&\frac{(n-2)(4\pi)^{\frac{n}{2}}}{(\frac{n}{2}-1)!}\int_{M}
2^n\bigg(-\frac{1}{12}s-\frac{1}{4}(\lambda^2_1+\lambda^2_2)|v^*|^2\bigg)d{\rm Vol_{M}}\\
{\rm Wres}(D^*_2D_1)^{-\frac{n-2}{2}}
&=&\frac{(n-2)(4\pi)^{\frac{n}{2}}}{(\frac{n}{2}-1)!}\int_{M}
\bigg[2^n\bigg(-\frac{1}{12}s
-\frac{\lambda^2_1+\lambda^2_2-2n\lambda_1\lambda_2+4\lambda_1\lambda_2}{8}|v|^2\bigg)\nonumber\\
&&+\frac{1}{2}{\rm tr}[\lambda_2\nabla^{TM}_{e_{j}}(\varepsilon(v^*))c(e_{j})-\lambda_1c(e_{j})\nabla^{TM}_{e_{j}}(l(v))]
\bigg]d{\rm Vol_{M}}.
\end{eqnarray}
where $s$ is the scalar curvature.
\end{thm}

\section{A Kastler-Kalau-Walze type theorem for $4$-dimensional manifolds with boundary}
In this section, we prove the Kastler-Kalau-Walze type theorem for $4$-dimensional oriented compact manifold with boundary about statistical de Rham Hodge Operators.

Denote by $M$ a $n$-dimensional manifold with boundary $\partial M$. We assume that $M$ is compact and oriented.
Let $\mathcal{B}$ be Boutet de Monvel's algebra, we now recall the main theorem in \cite{FGLS,Wa3}.
\begin{thm}\label{th:32}{\rm\cite{FGLS}}{\bf(Fedosov-Golse-Leichtnam-Schrohe)}
 Let $X$ and $\partial X$ be connected, ${\rm dim}X=n\geq3$,
 $A=\left(\begin{array}{lcr}\pi^+P+G &   K \\
T &  S    \end{array}\right)$ $\in \mathcal{B}$ , and denote by $p$, $b$ and $s$ the local symbols of $P,G$ and $S$ respectively.
 Define:
 \begin{eqnarray}
{\rm{\widetilde{Wres}}}(A)&=&\int_X\int_{\bf S}{\rm{trace}}_E\left[p_{-n}(x,\xi)\right]\sigma(\xi)dx \nonumber\\
&&+2\pi\int_ {\partial X}\int_{\bf S'}\left\{{\rm trace}_E\left[({\rm{tr}}b_{-n})(x',\xi')\right]+{\rm{trace}}
_F\left[s_{1-n}(x',\xi')\right]\right\}\sigma(\xi')dx',
\end{eqnarray}
Then\\

~~a) ${\rm \widetilde{Wres}}([A,B])=0 $, for any
$A,B\in\mathcal{B}$;\\

~~b) It is a unique continuous trace on
$\mathcal{B}/\mathcal{B}^{-\infty}$.
\end{thm}

\begin{defn}{\rm\cite{Wa3} }
Lower dimensional volumes of spin manifolds with boundary are defined by
\begin{equation}
{\rm Vol}^{(p_1,p_2)}_nM:= \widetilde{{\rm Wres}}[\pi^+D^{-p_1}\circ\pi^+D^{-p_2}].
\end{equation}
\end{defn}
By (2.1.4)-(2.1.8) in \cite{Wa3}, we get
\begin{equation}
\widetilde{{\rm Wres}}[\pi^+D^{-p_1}\circ\pi^+D^{-p_2}]=\int_M\int_{|\xi|=1}{\rm
trace}_{\wedge^*T^*M}[\sigma_{-n}(D^{-p_1-p_2})]\sigma(\xi)dx+\int_{\partial M}\Phi,
\end{equation}
and
\begin{eqnarray}
\Phi &=&\int_{|\xi'|=1}\int^{+\infty}_{-\infty}\sum^{\infty}_{j, k=0}\sum\frac{(-i)^{|\alpha|+j+k+1}}{\alpha!(j+k+1)!}
\times {\rm trace}_{\wedge^*T^*M}[\partial^j_{x_n}\partial^\alpha_{\xi'}\partial^k_{\xi_n}\sigma^+_{r}(D^{-p_1})(x',0,\xi',\xi_n)
\nonumber\\
&&\times\partial^\alpha_{x'}\partial^{j+1}_{\xi_n}\partial^k_{x_n}\sigma_{l}(D^{-p_2})(x',0,\xi',\xi_n)]d\xi_n\sigma(\xi')dx',
\end{eqnarray}
 where the sum is taken over $r+l-k-|\alpha|-j-1=-n,~~r\leq -p_1,l\leq -p_2$.


For any fixed point $x_0\in\partial M$, we choose the normal coordinates
$U$ of $x_0$ in $\partial M$ (not in $M$) and compute $\Phi(x_0)$ in the coordinates $\widetilde{U}=U\times [0,1)\subset M$ and the
metric $\frac{1}{h(x_n)}g^{\partial M}+dx_n^2.$ The dual metric of $g^M$ on $\widetilde{U}$ is ${h(x_n)}g^{\partial M}+dx_n^2.$  Write
$g^M_{ij}=g^M(\frac{\partial}{\partial x_i},\frac{\partial}{\partial x_j});~ g_M^{ij}=g^M(dx_i,dx_j)$, then

\begin{equation}
[g^M_{i,j}]= \left[\begin{array}{lcr}
  \frac{1}{h(x_n)}[g_{i,j}^{\partial M}]  & 0  \\
   0  &  1
\end{array}\right];~~~
[g_M^{i,j}]= \left[\begin{array}{lcr}
  h(x_n)[g^{i,j}_{\partial M}]  & 0  \\
   0  &  1
\end{array}\right],
\end{equation}
and
\begin{equation}
\partial_{x_s}g_{ij}^{\partial M}(x_0)=0, 1\leq i,j\leq n-1; ~~~g_{ij}^M(x_0)=\delta_{ij}.
\end{equation}
We will give following three lemmas as computation tools.
\begin{lem}{\rm \cite{Wa3}}\label{le:32}
With the metric $g^{M}$ on $M$ near the boundary
\begin{eqnarray}
\partial_{x_j}(|\xi|_{g^M}^2)(x_0)&=&\left\{
       \begin{array}{c}
        0,  ~~~~~~~~~~ ~~~~~~~~~~ ~~~~~~~~~~~~~{\rm if }~j<n, \\[2pt]
       h'(0)|\xi'|^{2}_{g^{\partial M}},~~~~~~~~~~~~~~~~~~~~{\rm if }~j=n;
       \end{array}
    \right. \\
\partial_{x_j}[c(\xi)](x_0)&=&\left\{
       \begin{array}{c}
      0,  ~~~~~~~~~~ ~~~~~~~~~~ ~~~~~~~~~~~~~{\rm if }~j<n,\\[2pt]
\partial x_{n}(c(\xi'))(x_{0}), ~~~~~~~~~~~~~~~~~{\rm if }~j=n,
       \end{array}
    \right.
\end{eqnarray}
where $\xi=\xi'+\xi_{n}dx_{n}$.
\end{lem}
\begin{lem}{\rm \cite{Wa3}}\label{le:32}With the metric $g^{M}$ on $M$ near the boundary
\begin{eqnarray}
\omega_{s,t}(e_i)(x_0)&=&\left\{
       \begin{array}{c}
        \omega_{n,i}(e_i)(x_0)=\frac{1}{2}h'(0),  ~~~~~~~~~~ ~~~~~~~~~~~{\rm if }~s=n,t=i,i<n; \\[2pt]
       \omega_{i,n}(e_i)(x_0)=-\frac{1}{2}h'(0),~~~~~~~~~~~~~~~~~~~{\rm if }~s=i,t=n,i<n;\\[2pt]
    \omega_{s,t}(e_i)(x_0)=0,~~~~~~~~~~~~~~~~~~~~~~~~~~~other~cases~~~~~~~~~,
       \end{array}
    \right.
\end{eqnarray}
where $(\omega_{s,t})$ denotes the connection matrix of Levi-Civita connection $\nabla^L$.
\end{lem}
\begin{lem}{\rm \cite{Wa3}}~{\it When $i<n$, then }
$$\Gamma^n_{ii}(x_0)=\frac{1}{2}h'(0);~\Gamma^i_{ni}(x_0)=-\frac{1}{2}h'(0);~\Gamma^i_{in}(x_0)=-\frac{1}{2}h'(0),$$
\noindent {\it in other cases}, $ \Gamma_{st}^i(x_0)=0.$
\end{lem}

By (3.3) and (3.4), we firstly compute
\begin{equation}
\widetilde{{\rm Wres}}[\pi^+D_2^{-1}\circ\pi^+D_1^{-1}]=\int_M\int_{|\xi|=1}{\rm
trace}_{\wedge^*T^*M}[\sigma_{-4}(D_1D_2)^{-1}]\sigma(\xi)dx+\int_{\partial M}\Phi_1,
\end{equation}
where
\begin{eqnarray}
\Phi_1 &=&\int_{|\xi'|=1}\int^{+\infty}_{-\infty}\sum^{\infty}_{j, k=0}\sum\frac{(-i)^{|\alpha|+j+k+1}}{\alpha!(j+k+1)!}
\times {\rm trace}_{\wedge^*T^*M}[\partial^j_{x_n}\partial^\alpha_{\xi'}\partial^k_{\xi_n}\sigma^+_{r}
(D_2^{-1})(x',0,\xi',\xi_n)
\nonumber\\
&&\times\partial^\alpha_{x'}\partial^{j+1}_{\xi_n}
\partial^k_{x_n}\sigma_{l}(D_1^{-1})(x',0,\xi',\xi_n)]d\xi_n\sigma(\xi')dx',
\end{eqnarray}
and the sum is taken over $r+l-k-j-|\alpha|=-3,~~r\leq -1,l\leq-1$.\\

Locally we can use Theorem 2.2 (2.42) to compute the interior of $\widetilde{{\rm Wres}}[\pi^+D_2^{-1}\circ\pi^+D_1^{-1}]$, we have
\begin{eqnarray}
\int_M\int_{|\xi|=1}{\rm
trace}_{\wedge^*T^*M}[\sigma_{-4}(D_1D_2)^{-1}]\sigma(\xi)dx=32\pi^2\int_{M}
\bigg(-\frac{4}{3}s-4(\lambda^2_1+\lambda^2_2)|v|^2\bigg)d{\rm Vol_{M}}
\end{eqnarray}

So we only need to compute $\int_{\partial M} \Phi_1$. Let us now turn to compute the symbols of some operators.
By (2.10)-(2.14), then we have the following symbols of some operators.
\begin{lem} The following identities hold:
\begin{eqnarray}
\sigma_1(D_j)&=&\sigma_1(D^*_j)=ic(\xi),~(j=1,2); \nonumber\\ \sigma_0(D_j)&=&\frac{1}{4}\sum_{i,s,t}\omega_{s,t}(e_i)c(e_i)
\bar{c}(e_s)\bar{c}(e_t)
-\frac{1}{4}\sum_{i,s,t}\omega_{s,t}(e_i)c(e_i)
c(e_s)c(e_t)+\lambda_jl(v),~(j=1,2); \nonumber\\
\sigma_0(D_j^*)&=&\frac{1}{4}\sum_{i,s,t}\omega_{s,t}(e_i)c(e_i)
\bar{c}(e_s)\bar{c}(e_t)
-\frac{1}{4}\sum_{i,s,t}\omega_{s,t}(e_i)c(e_i)
c(e_s)c(e_t)+\lambda_j\varepsilon(v^*),~(j=1,2).
\end{eqnarray}
\end{lem}

Write
 \begin{eqnarray}
D_x^{\alpha}&=&(-i)^{|\alpha|}\partial_x^{\alpha};
~\sigma(D)=p_1+p_0;
~\sigma(D^{-1})=\sum^{\infty}_{j=1}q_{-j}.
\end{eqnarray}

By the composition formula of pseudodifferential operators, we have
\begin{eqnarray}
1&=&\sigma(D\circ D^{-1})\nonumber\\
&=&\sum_{\alpha}\frac{1}{\alpha!}\partial^{\alpha}_{\xi}[\sigma(D)]
D^{\alpha}_{x}[\sigma(D^{-1})]\nonumber\\
&=&(p_1+p_0)(q_{-1}+q_{-2}+q_{-3}+\cdots)\nonumber\\
& &~~~+\sum_j(\partial_{\xi_j}p_1+\partial_{\xi_j}p_0)(
D_{x_j}q_{-1}+D_{x_j}q_{-2}+D_{x_j}q_{-3}+\cdots)\nonumber\\
&=&p_1q_{-1}+(p_1q_{-2}+p_0q_{-1}+\sum_j\partial_{\xi_j}p_1D_{x_j}q_{-1})+\cdots,
\end{eqnarray}
so
\begin{equation}
q_{-1}=p_1^{-1};~q_{-2}=-p_1^{-1}[p_0p_1^{-1}+\sum_j\partial_{\xi_j}p_1D_{x_j}(p_1^{-1})].
\end{equation}

By Lemma 3.6, we have some symbols of operators.
\begin{lem} The following identities hold:
\begin{eqnarray}
\sigma_{-1}(D_j^{-1})&=&\sigma_{-1}((D_j^*)^{-1})=\frac{ic(\xi)}{|\xi|^2},~(j=1,2);\nonumber\\
\sigma_{-2}(D_j^{-1})&=&\frac{c(\xi)\sigma_{0}(D_j)c(\xi)}{|\xi|^4}+\frac{c(\xi)}{|\xi|^6}\sum_jc(dx_j)
\Big[\partial_{x_j}(c(\xi))|\xi|^2-c(\xi)\partial_{x_j}(|\xi|^2)\Big],~(j=1,2);\nonumber\\
\sigma_{-2}((D_j^*)^{-1})&=&\frac{c(\xi)\sigma_{0}(D_j^*)c(\xi)}{|\xi|^4}+\frac{c(\xi)}{|\xi|^6}\sum_jc(dx_j)
\Big[\partial_{x_j}(c(\xi))|\xi|^2-c(\xi)\partial_{x_j}(|\xi|^2)\Big],~(j=1,2). \end{eqnarray}
\end{lem}

From the remark above, now we can compute $\Phi_1$ (see formula (3.11) for the definition of $\Phi_1$). We use ${\rm tr}$ as shorthand of ${\rm trace}$. Since $n=4$, then ${\rm tr}_{\wedge^*T^*M}[{\rm \texttt{id}}]=16$, since the sum is taken over $
r+l-k-j-|\alpha|=-3,~~r\leq -1,l\leq-1,$ we have the following five cases:

~\\
\noindent  {\bf case 1)~I)}~$r=-1,~l=-1,~k=j=0,~|\alpha|=1$\\

\noindent By (3.11), we get
\begin{equation}
{\rm case~1)~I)}=-\int_{|\xi'|=1}\int^{+\infty}_{-\infty}\sum_{|\alpha|=1}
 {\rm tr}[\partial^\alpha_{\xi'}\pi^+_{\xi_n}\sigma_{-1}(D_2^{-1})\times
 \partial^\alpha_{x'}\partial_{\xi_n}\sigma_{-1}(D_1^{-1})](x_0)d\xi_n\sigma(\xi')dx'.
\end{equation}
By Lemma 3.3, for $i<n$, then
\begin{equation}\partial_{x_i}\left(\frac{ic(\xi)}{|\xi|^2}\right)(x_0)=
\frac{i\partial_{x_i}[c(\xi)](x_0)}{|\xi|^2}
-\frac{ic(\xi)\partial_{x_i}(|\xi|^2)(x_0)}{|\xi|^4}=0,
\end{equation}
\noindent so {\rm case~1)~I)} vanishes.\\

 \noindent  {\bf case 1)~II)}~$r=-1,~l=-1,~k=|\alpha|=0,~j=1$\\

\noindent By (3.11), we get
\begin{equation}
{\rm case~1)~II)}=-\frac{1}{2}\int_{|\xi'|=1}\int^{+\infty}_{-\infty} {\rm
tr} [\partial_{x_n}\pi^+_{\xi_n}\sigma_{-1}(D_2^{-1})\times
\partial_{\xi_n}^2\sigma_{-1}(D_1^{-1})](x_0)d\xi_n\sigma(\xi')dx'.
\end{equation}
\noindent By Lemma 3.7, we have\\
\begin{eqnarray}\partial^2_{\xi_n}\sigma_{-1}(D_1^{-1})(x_0)=i\left(-\frac{6\xi_nc(dx_n)+2c(\xi')}
{|\xi|^4}+\frac{8\xi_n^2c(\xi)}{|\xi|^6}\right);
\end{eqnarray}
\begin{eqnarray}
\partial_{x_n}\sigma_{-1}(D_2^{-1})(x_0)=\frac{i\partial_{x_n}c(\xi')(x_0)}{|\xi|^2}-\frac{ic(\xi)|\xi'|^2h'(0)}{|\xi|^4}.
\end{eqnarray}
By (2.16), (2.17) and the Cauchy integral formula we have
\begin{eqnarray}
\pi^+_{\xi_n}\left[\frac{c(\xi)}{|\xi|^4}\right](x_0)|_{|\xi'|=1}&=&\pi^+_{\xi_n}\left[\frac{c(\xi')+\xi_nc(dx_n)}{(1+\xi_n^2)^2}\right]\nonumber\\
&=&\frac{1}{2\pi i}{\rm lim}_{u\rightarrow
0^-}\int_{\Gamma^+}\frac{\frac{c(\xi')+\eta_nc(dx_n)}{(\eta_n+i)^2(\xi_n+iu-\eta_n)}}
{(\eta_n-i)^2}d\eta_n\nonumber\\
&=&-\frac{(i\xi_n+2)c(\xi')+ic(dx_n)}{4(\xi_n-i)^2}.
\end{eqnarray}
Similarly, we have
\begin{eqnarray}
\pi^+_{\xi_n}\left[\frac{i\partial_{x_n}c(\xi')}{|\xi|^2}\right](x_0)|_{|\xi'|=1}=\frac{\partial_{x_n}[c(\xi')](x_0)}{2(\xi_n-i)}.
\end{eqnarray}
then\\
\begin{eqnarray}\pi^+_{\xi_n}\partial_{x_n}\sigma_{-1}(D_2^{-1})|_{|\xi'|=1}
=\frac{\partial_{x_n}[c(\xi')](x_0)}{2(\xi_n-i)}+ih'(0)
\left[\frac{(i\xi_n+2)c(\xi')+ic(dx_n)}{4(\xi_n-i)^2}\right].
\end{eqnarray}
\noindent By the relation of the Clifford action and ${\rm tr}{AB}={\rm tr}{BA}$, we have the equalities:\\
\begin{eqnarray*}{\rm tr}[c(\xi')c(dx_n)]=0;~~{\rm tr}[c(dx_n)^2]=-16;~~{\rm tr}[c(\xi')^2](x_0)|_{|\xi'|=1}=-16;
\end{eqnarray*}
\begin{eqnarray}
{\rm tr}[\partial_{x_n}c(\xi')c(dx_n)]=0;~~{\rm tr}[\partial_{x_n}c(\xi')c(\xi')](x_0)|_{|\xi'|=1}=-8h'(0);~~{\rm tr}[\bar{c}(e_i)\bar{c}(e_j)c(e_k)c(e_l)]=0(i\neq j).
\end{eqnarray}
By (3.26) and a direct computation, we have
\begin{eqnarray}
&&h'(0){\rm tr}\bigg[\frac{(i\xi_n+2)c(\xi')+ic(dx_n)}{4(\xi_n-i)^2}\times
\bigg(\frac{6\xi_nc(dx_n)+2c(\xi')}{(1+\xi_n^2)^2}
-\frac{8\xi_n^2[c(\xi')+\xi_nc(dx_n)]}{(1+\xi_n^2)^3}\bigg)
\bigg](x_0)|_{|\xi'|=1}\nonumber\\
&=&-16h'(0)\frac{-2i\xi_n^2-\xi_n+i}{(\xi_n-i)^4(\xi_n+i)^3}.
\end{eqnarray}
Similarly, we have
\begin{eqnarray}
&&-i{\rm
tr}\bigg[\bigg(\frac{\partial_{x_n}[c(\xi')](x_0)}{2(\xi_n-i)}\bigg)
\times\bigg(\frac{6\xi_nc(dx_n)+2c(\xi')}{(1+\xi_n^2)^2}-\frac{8\xi_n^2[c(\xi')+\xi_nc(dx_n)]}
{(1+\xi_n^2)^3}\bigg)\bigg](x_0)|_{|\xi'|=1}\nonumber\\
&=&-8ih'(0)\frac{3\xi_n^2-1}{(\xi_n-i)^4(\xi_n+i)^3}.
\end{eqnarray}
Then\\
\begin{eqnarray*}
{\rm case~1)~II)}&=&-\int_{|\xi'|=1}\int^{+\infty}_{-\infty}\frac{4ih'(0)(\xi_n-i)^2}
{(\xi_n-i)^4(\xi_n+i)^3}d\xi_n\sigma(\xi')dx'\\
&=&-4ih'(0)\Omega_3\int_{\Gamma^+}\frac{1}{(\xi_n-i)^2(\xi_n+i)^3}d\xi_ndx'\\
&=&-4ih'(0)\Omega_32\pi i[\frac{1}{(\xi_n+i)^3}]'|_{\xi_n=i}dx'\\
&=&-\frac{3}{2}\pi h'(0)\Omega_3dx'.
\end{eqnarray*}
where ${\rm \Omega_{3}}$ is the canonical volume of $S^{3}.$\\

\noindent  {\bf case 1)~III)}~$r=-1,~l=-1,~j=|\alpha|=0,~k=1$\\

\noindent By (3.11), we get
\begin{equation}
{\rm case~1)~III)}=-\frac{1}{2}\int_{|\xi'|=1}\int^{+\infty}_{-\infty}
{\rm tr} [\partial_{\xi_n}\pi^+_{\xi_n}\sigma_{-1}(D_2^{-1})\times
\partial_{\xi_n}\partial_{x_n}\sigma_{-1}(D_1^{-1})](x_0)d\xi_n\sigma(\xi')dx'.
\end{equation}
\noindent By Lemma 3.7, we have\\
\begin{eqnarray}
\partial_{\xi_n}\partial_{x_n}\sigma_{-1}(D_1^{-1})(x_0)|_{|\xi'|=1}
&=&-ih'(0)\left[\frac{c(dx_n)}{|\xi|^4}-4\xi_n\frac{c(\xi')
+\xi_nc(dx_n)}{|\xi|^6}\right]-\frac{2\xi_ni\partial_{x_n}c(\xi')(x_0)}{|\xi|^4};
\end{eqnarray}
\begin{eqnarray}
\partial_{\xi_n}\pi^+_{\xi_n}\sigma_{-1}(D_2^{-1})(x_0)|_{|\xi'|=1}&=&-\frac{c(\xi')+ic(dx_n)}{2(\xi_n-i)^2}.
\end{eqnarray}
Similar to {\rm case~1)~II)}, we have\\
\begin{eqnarray}
&&{\rm tr}\left\{\frac{c(\xi')+ic(dx_n)}{2(\xi_n-i)^2}\times
ih'(0)\left[\frac{c(dx_n)}{|\xi|^4}-4\xi_n\frac{c(\xi')+\xi_nc(dx_n)}{|\xi|^6}\right]\right\}
=8h'(0)\frac{i-3\xi_n}{(\xi_n-i)^4(\xi_n+i)^3}
\end{eqnarray}
and
\begin{eqnarray}
{\rm tr}\left[\frac{c(\xi')+ic(dx_n)}{2(\xi_n-i)^2}\times
\frac{2\xi_ni\partial_{x_n}c(\xi')(x_0)}{|\xi|^4}\right]
=\frac{-8ih'(0)\xi_n}{(\xi_n-i)^4(\xi_n+i)^2}.
\end{eqnarray}
So we have
\begin{eqnarray}
{\rm case~1)~III)}&=&-\int_{|\xi'|=1}\int^{+\infty}_{-\infty}\frac{h'(0)4(i-3\xi_n)}
{(\xi_n-i)^4(\xi_n+i)^3}d\xi_n\sigma(\xi')dx'
-\int_{|\xi'|=1}\int^{+\infty}_{-\infty}\frac{h'(0)4i\xi_n}
{(\xi_n-i)^4(\xi_n+i)^2}d\xi_n\sigma(\xi')dx'\nonumber\\
&=&-h'(0)\Omega_3\frac{2\pi i}{3!}[\frac{4(i-3\xi_n)}{(\xi_n+i)^3}]^{(3)}|_{\xi_n=i}dx'+h'(0)\Omega_3\frac{2\pi i}{3!}[\frac{4i\xi_n}{(\xi_n+i)^2}]^{(3)}|_{\xi_n=i}dx'\nonumber\\
&=&\frac{3}{2}\pi h'(0)\Omega_3dx'.
\end{eqnarray}

\noindent  {\bf case 2)}~$r=-2,~l=-1,~k=j=|\alpha|=0$\\

\noindent By (3.11), we get
\begin{eqnarray}
{\rm case~2)}&=&-i\int_{|\xi'|=1}\int^{+\infty}_{-\infty}{\rm tr} [\pi^+_{\xi_n}\sigma_{-2}(D_2^{-1})\times
\partial_{\xi_n}\sigma_{-1}(D_1^{-1})](x_0)d\xi_n\sigma(\xi')dx'.
\end{eqnarray}
 By Lemma 3.7 we have\\
\begin{eqnarray}
\sigma_{-2}(D_2^{-1})(x_0)=\frac{c(\xi)\sigma_{0}(D_2)(x_0)c(\xi)}{|\xi|^4}+\frac{c(\xi)}{|\xi|^6}c(dx_n)
[\partial_{x_n}[c(\xi')](x_0)|\xi|^2-c(\xi)h'(0)|\xi|^2_{\partial
M}].
\end{eqnarray}
where
\begin{eqnarray}
\sigma_{0}(D_2)(x_0)&=&\frac{1}{4}\sum_{s,t,i}\omega_{s,t}(e_i)
(x_{0})c(e_i)\bar{c}(e_s)\bar{c}(e_t)-\frac{1}{4}\sum_{s,t,i}\omega_{s,t}(e_i)(x_{0})c(e_i)c(e_s)c(e_t))
+\lambda_2l(v).
\end{eqnarray}
Write
\begin{eqnarray}
A(x_0)&=&\frac{1}{4}\sum_{s,t,i}\omega_{s,t}(e_i)
(x_{0})c(e_i)\bar{c}(e_s)\bar{c}(e_t);
B(x_0)=-\frac{1}{4}\sum_{s,t,i}\omega_{s,t}(e_i)(x_{0})c(e_i)c(e_s)c(e_t)).
\end{eqnarray}
Then
\begin{eqnarray}
&&\pi^+_{\xi_n}\sigma_{-2}(D_2^{-1}(x_0))|_{|\xi'|=1}\nonumber\\
&=&\pi^+_{\xi_n}\Big[\frac{c(\xi)A(x_0)c(\xi)}{(1+\xi_n^2)^2}\Big]+\pi^+_{\xi_n}
\Big[\lambda_2\frac{c(\xi)(l(v)(x_0))c(\xi)}{(1+\xi_n^2)^2}\Big]
+\pi^+_{\xi_n}\Big[\frac{c(\xi)B(x_0)c(\xi)
+c(\xi)c(dx_n)\partial_{x_n}[c(\xi')](x_0)}{(1+\xi_n^2)^2}\nonumber\\
&&-h'(0)\frac{c(\xi)c(dx_n)c(\xi)}{(1+\xi_n^{2})^3}\Big].
\end{eqnarray}
By direct calculation we have
\begin{eqnarray}
&&\pi^+_{\xi_n}\Big[\frac{c(\xi)A(x_0)c(\xi)}{(1+\xi_n^2)^2}\Big]\nonumber\\
&=&\pi^+_{\xi_n}\Big[\frac{c(\xi')A(x_0)c(\xi')}{(1+\xi_n^2)^2}\Big]
+\pi^+_{\xi_n}\Big[ \frac{\xi_nc(\xi')A(x_0)c(dx_{n})}{(1+\xi_n^2)^2}\Big]
+\pi^+_{\xi_n}\Big[\frac{\xi_nc(dx_{n})A(x_0)c(\xi')}{(1+\xi_n^2)^2}\Big]
\nonumber\\
&&+\pi^+_{\xi_n}\Big[\frac{\xi_n^{2}c(dx_{n})A(x_0)c(dx_{n})}{(1+\xi_n^2)^2}\Big]\nonumber\\
&=&-\frac{c(\xi')A(x_0)c(\xi')(2+i\xi_{n})}{4(\xi_{n}-i)^{2}}
+\frac{ic(\xi')A(x_0)c(dx_{n})}{4(\xi_{n}-i)^{2}}
+\frac{ic(dx_{n})A(x_0)c(\xi')}{4(\xi_{n}-i)^{2}}
\nonumber\\
&&+\frac{-i\xi_{n}c(dx_{n})A(x_0)c(dx_{n})}{4(\xi_{n}-i)^{2}}.
\end{eqnarray}
Since
\begin{eqnarray}
c(dx_n)A(x_0)
&=&-\frac{1}{4}h'(0)\sum^{n-1}_{i=1}c(e_i)
\bar{c}(e_i)c(e_n)\bar{c}(e_n),
\end{eqnarray}
by the relation of the Clifford action and ${\rm tr}{AB}={\rm tr }{BA}$,  we have the equalities:\\
\begin{eqnarray}
{\rm tr}[c(e_i)
\bar{c}(e_i)c(e_n)
\bar{c}(e_n)]&=&0~~(i<n);~~
{\rm tr}[Ac(dx_n)]=0;~~{\rm tr }[\bar{c}(\xi')c(dx_{n})]=0;
\end{eqnarray}
Since
\begin{eqnarray}
\partial_{\xi_n}\sigma_{-1}(D_v^{-1})=\partial_{\xi_n}q_{-1}(x_0)|_{|\xi'|=1}=i\left[\frac{c(dx_n)}{1+\xi_n^2}-\frac{2\xi_nc(\xi')+2\xi_n^2c(dx_n)}{(1+\xi_n^2)^2}\right],
\end{eqnarray}
By (3.40), (3.42) and (3.43), we have
\begin{eqnarray}
&&{\rm tr }[\pi^+_{\xi_n}\Big[\frac{c(\xi)A(x_0)c(\xi)}{(1+\xi_n^2)^2}\Big]
\times\partial_{\xi_n}\sigma_{-1}(D_1^{-1})(x_0)]|_{|\xi'|=1}\nonumber\\
&=&\frac{1}{2(1+\xi_n^2)^2}{\rm tr }[c(\xi')A(x_0)]
+\frac{i}{2(1+\xi_n^2)^2}{\rm tr }[c(dx_n)A(x_0)]\nonumber\\
&=&\frac{1}{2(1+\xi_n^2)^2}{\rm tr }[c(\xi')A(x_0)].
\end{eqnarray}
We note that $i<n,~\int_{|\xi'|=1}\{\xi_{i_{1}}\xi_{i_{2}}\cdots\xi_{i_{2d+1}}\}\sigma(\xi')=0$,
so ${\rm tr }[c(\xi')A(x_0)]$ has no contribution for computing {\rm case~2)}.

By direct calculation we have
\begin{eqnarray}
\pi^+_{\xi_n}\Big[\frac{c(\xi)B(x_0)c(\xi)+c(\xi)c(dx_n)
\partial_{x_n}[c(\xi')](x_0)}{(1+\xi_n^2)^2}\Big]
-h'(0)\pi^+_{\xi_n}\Big[\frac{c(\xi)c(dx_n)c(\xi)}{(1+\xi_n)^3}\Big]:= P_1-P_2,
\end{eqnarray}
where
\begin{eqnarray}
P_1&=&\frac{-1}{4(\xi_n-i)^2}[(2+i\xi_n)c(\xi')b_0^{2}(x_0)c(\xi')+i\xi_nc(dx_n)b_0^{2}(x_0)c(dx_n)\nonumber\\
&&+(2+i\xi_n)c(\xi')c(dx_n)\partial_{x_n}c(\xi')+ic(dx_n)b_0^{2}(x_0)c(\xi')
+ic(\xi')b_0^{2}(x_0)c(dx_n)-i\partial_{x_n}c(\xi')]
\end{eqnarray}
and
\begin{eqnarray}
P_2&=&\frac{h'(0)}{2}\left[\frac{c(dx_n)}{4i(\xi_n-i)}+\frac{c(dx_n)-ic(\xi')}{8(\xi_n-i)^2}
+\frac{3\xi_n-7i}{8(\xi_n-i)^3}[ic(\xi')-c(dx_n)]\right].
\end{eqnarray}
By (3.43) and (3.46), we have\\
\begin{eqnarray}{\rm tr }[P_1\times\partial_{\xi_n}\sigma_{-1}(D_1^{-1})]|_{|\xi'|=1}=
\frac{-6ih'(0)}{(1+\xi_n^2)^2}+2h'(0)\frac{\xi_n^2-i\xi_n-2}{(\xi_n-i)(1+\xi_n^2)^2},
\end{eqnarray}
By (3.43) and (3.47), we have
\begin{eqnarray}{\rm tr }[P_2\times\partial_{\xi_n}\sigma_{-1}(D_1^{-1})]|_{|\xi'|=1}
&=&\frac{i}{2}h'(0)\frac{-i\xi_n^2-\xi_n+4i}{4(\xi_n-i)^3(\xi_n+i)^2}{\rm tr}[ \texttt{id}]\nonumber\\
&=&8ih'(0)\frac{-i\xi_n^2-\xi_n+4i}{4(\xi_n-i)^3(\xi_n+i)^2}.
\end{eqnarray}
By (3.48) and (3.49), we have
\begin{eqnarray}
&&-i\int_{|\xi'|=1}\int^{+\infty}_{-\infty}{\rm tr} [(P_1-P_2)\times
\partial_{\xi_n}\sigma_{-1}(D_1^{-1})](x_0)d\xi_n\sigma(\xi')dx'\nonumber\\
&=&-\Omega_3\int_{\Gamma^+}\frac{8[-\frac{3}{4}h'(0)](\xi_n-i)+ih'(0)}{(\xi_n-i)^3(\xi_n+i)^2}d\xi_ndx'\nonumber\\
&=&\frac{9}{2}\pi h'(0)\Omega_3dx'.
\end{eqnarray}
Similar to (3.44), we have
\begin{eqnarray}
&&{\rm tr }[\pi^+_{\xi_n}\Big[\lambda_2\frac{c(\xi)l(v)(x_0)c(\xi)}{(1+\xi_n^2)^2}\Big]
\times\partial_{\xi_n}\sigma_{-1}(D_1)^{-1})(x_0)]|_{|\xi'|=1}\nonumber\\
&=&\frac{1}{2(1+\xi_n^2)^2}{\rm tr}[\lambda_2c(\xi')l(v)(x_0)]
+\frac{i}{2(1+\xi_n^2)^2}{\rm tr}[\lambda_2c(dx_n)l(v)(x_0)].
\end{eqnarray}
By the relation of the Clifford action and ${\rm tr}{AB}={\rm tr }{BA}$,  we have the equalities:\\
\begin{eqnarray}
{\rm tr}[c(dx_n)l(v)]=8\langle v, dx_n\rangle;
~~{\rm tr }[c(\xi')l(v)]=8\langle v, \xi'\rangle;
\end{eqnarray}
By (3.51) and (3.52), we have
\begin{eqnarray}
&&-i\int_{|\xi'|=1}\int^{+\infty}_{-\infty}{\rm tr} [\pi^+_{\xi_n}
\Big[\frac{\lambda_2c(\xi)(l(v))c(\xi)}{(1+\xi_n^2)^2}\Big]\times
\partial_{\xi_n}\sigma_{-1}(D_1^{-1})](x_0)d\xi_n\sigma(\xi')dx'\nonumber\\
&=&2\lambda_2\pi\langle v, dx_n\rangle\Omega_3dx'.
\end{eqnarray}
By (3.44), (3.50) and (3.53), we have\\
\begin{eqnarray}
{\rm case~2)}=[\frac{9}{2}\pi h'(0)+2\lambda_2\pi\langle v, dx_n\rangle]\Omega_3dx'.
\end{eqnarray}

\noindent {\bf  case 3)}~$r=-1,~l=-2,~k=j=|\alpha|=0$\\
By (3.11), we get
\begin{equation}
{\rm case~ 3)}=-i\int_{|\xi'|=1}\int^{+\infty}_{-\infty}{\rm tr} [\pi^+_{\xi_n}\sigma_{-1}(D_2^{-1})\times
\partial_{\xi_n}\sigma_{-2}(D_1^{-1})](x_0)d\xi_n\sigma(\xi')dx'.
\end{equation}
By (2.16) and Lemma 3.7, we have
\begin{equation}
\pi^+_{\xi_n}\sigma_{-1}(D_2^{-1})|_{|\xi'|=1}=\frac{c(\xi')+ic(dx_n)}{2(\xi_n-i)}.
\end{equation}
By (3.36), (3.37) and (3.38), we have
\begin{eqnarray}
&&\partial_{\xi_n}\sigma_{-2}(D_1^{-1})(x_0)|_{|\xi'|=1}\nonumber\\
&=&
\partial_{\xi_n}\bigg\{\frac{c(\xi)[A(x_0)+B(x_0)
+(\lambda_1l(v)(x_{0}))]c(\xi)}{|\xi|^4}
+\frac{c(\xi)}{|\xi|^6}c(dx_n)[\partial_{x_n}[c(\xi')](x_0)|\xi|^2-c(\xi)h'(0)]\bigg\}\nonumber\\
&=&\partial_{\xi_n}\bigg\{\frac{c(\xi)A(x_0)]c(\xi)}{|\xi|^4}
+\frac{c(\xi)}{|\xi|^6}c(dx_n)[\partial_{x_n}[c(\xi')](x_0)|\xi|^2-c(\xi)h'(0)]\bigg\}
+\partial_{\xi_n}\frac{c(\xi)B(x_0)c(\xi)}{|\xi|^4}\nonumber\\
&&+\lambda_1\partial_{\xi_n}\frac{c(\xi)(l(v)(x_{0}))c(\xi)}{|\xi|^4}.
\end{eqnarray}
By direct calculation we have
\begin{eqnarray}
\partial_{\xi_n}\frac{c(\xi)A(x_0)c(\xi)}{|\xi|^4}=\frac{c(dx_n)A(x_0)c(\xi)}{|\xi|^4}
+\frac{c(\xi)A(x_0)c(dx_n)}{|\xi|^4}
-\frac{4\xi_n c(\xi)A(x_0)c(\xi)}{|\xi|^6};
\end{eqnarray}
\begin{eqnarray}
\partial_{\xi_n}\frac{c(\xi)(l(v)(x_{0}))c(\xi)}{|\xi|^4}
&=&\frac{c(dx_n)(l(v)(x_{0}))c(\xi)}{|\xi|^4}
+\frac{c(\xi)(l(v)(x_{0}))c(dx_n)}{|\xi|^4}-\frac{4\xi_n c(\xi)(l(v)(x_{0}))c(\xi)}{|\xi|^4}.
\end{eqnarray}
Write $$P_3=\frac{c(\xi)B(x_0)c(\xi)}{|\xi|^4}
+\frac{c(\xi)}{|\xi|^6}c(dx_n)[\partial_{x_n}[c(\xi')](x_0)|\xi|^2-c(\xi)h'(0)],$$ then
\begin{eqnarray}
\partial_{\xi_n}(P_3)&=&\frac{1}{(1+\xi_n^2)^3}\bigg[(2\xi_n-2\xi_n^3)c(dx_n)Bc(dx_n)
+(1-3\xi_n^2)c(dx_n)Bc(\xi')\nonumber\\
&&+(1-3\xi_n^2)c(\xi')Bc(dx_n)
-4\xi_nc(\xi')Bc(\xi')
+(3\xi_n^2-1)\partial_{x_n}c(\xi')\nonumber\\
&&-4\xi_nc(\xi')c(dx_n)\partial_{x_n}c(\xi')
+2h'(0)c(\xi')+2h'(0)\xi_nc(dx_n)\bigg]\nonumber\\
&&+6\xi_nh'(0)\frac{c(\xi)c(dx_n)c(\xi)}{(1+\xi^2_n)^4};
\end{eqnarray}
By (3.56) and (3.58), we have
\begin{eqnarray}
&&{\rm tr}[\pi^+_{\xi_n}\sigma_{-1}(D_2^{-1})\times
\partial_{\xi_n}\frac{c(\xi)Ac(\xi)}
{|\xi|^4}](x_0)|_{|\xi'|=1}\nonumber\\
&=&\frac{-1}{(\xi-i)(\xi+i)^3}{\rm tr}[c(\xi')A(x_0)]
+\frac{i}{(\xi-i)(\xi+i)^3}{\rm tr}[c(dx_n)A(x_0)].
\end{eqnarray}
By (3.42), we have
\begin{eqnarray}
{\rm tr}[\pi^+_{\xi_n}\sigma_{-1}(D_2^{-1})\times
\partial_{\xi_n}\frac{c(\xi)Ac(\xi)}
{|\xi|^4}](x_0)|_{|\xi'|=1}
=\frac{-1}{(\xi-i)(\xi+i)^3}{\rm tr}[c(\xi')A(x_0)].
\end{eqnarray}
We note that $i<n,~\int_{|\xi'|=1}\{\xi_{i_{1}}\xi_{i_{2}}\cdots\xi_{i_{2d+1}}\}\sigma(\xi')=0$,
so ${\rm tr }[c(\xi')A(x_0)]$ has no contribution for computing {\rm case~3)}.

By (3.56) and (3.60), we have
\begin{eqnarray}
{\rm tr}[\pi^+_{\xi_n}\sigma_{-1}(D_2^{-1})\times
\partial_{\xi_n}(P_3)](x_0)|_{|\xi'|=1}
=\frac{12h'(0)(i\xi^2_n+\xi_n-2i)}{(\xi-i)^3(\xi+i)^3}
+\frac{48h'(0)i\xi_n}{(\xi-i)^3(\xi+i)^4},
\end{eqnarray}
then
\begin{eqnarray}
-i\Omega_3\int_{\Gamma_+}\Big[\frac{12h'(0)(i\xi_n^2+\xi_n-2i)}
{(\xi_n-i)^3(\xi_n+i)^3}+\frac{48h'(0)i\xi_n}{(\xi_n-i)^3(\xi_n+i)^4}\Big]d\xi_ndx'=
-\frac{9}{2}\pi h'(0)\Omega_3dx'.
\end{eqnarray}
By (3.56) and (3.59), we have
\begin{eqnarray}
&&{\rm tr}[\pi^+_{\xi_n}\sigma_{-1}(D_2^{-1})\times
\partial_{\xi_n}\frac{c(\xi)(\lambda_1l(v))c(\xi)}
{|\xi|^4}](x_0)|_{|\xi'|=1}\nonumber\\
&=&\frac{-1}{(\xi-i)(\xi+i)^3}{\rm tr}[c(\xi')(\lambda_1l(v))(x_0)]
+\frac{i}{(\xi-i)(\xi+i)^3}{\rm tr}[c(dx_n)(\lambda_1l(v))(x_0)].
\end{eqnarray}
By (3.52), we have
\begin{eqnarray}
&&-i\int_{|\xi'|=1}\int^{+\infty}_{-\infty}{\rm tr}[\pi^+_{\xi_n}\sigma_{-1}(D_2^{-1})\times
\partial_{\xi_n}\frac{c(\xi)(\lambda_1l(v))c(\xi)}
{|\xi|^4}](x_0)d\xi_n\sigma(\xi')dx'\nonumber\\
&=&-i\int_{|\xi'|=1}\int^{+\infty}_{-\infty}\frac{i}{(\xi-i)(\xi+i)^3}\Big[{\rm tr}[c(dx_n)(\lambda_1l(v))(x_0)]+i{\rm tr}[c(\xi')(\lambda_1l(v))(x_0)]\Big]d\xi_n\sigma(\xi')dx'\nonumber\\
&=&-\frac{\pi}{4}\Big[{\rm tr}[c(dx_n)(\lambda_1l(v))]+i{\rm tr}[c(\xi')(\lambda_1l(v))(x_0)]\Big]\Omega_3dx'\nonumber\\
&=&-2\pi\lambda_1\langle v, dx_n\rangle\Omega_3dx'
.
\end{eqnarray}
So we have
\begin{eqnarray}
{\rm case~ 3)}=\Big[-\frac{9}{2}\pi h'(0)-2\lambda_1\pi\langle v,dx_n\rangle\Big]\Omega_3dx'.
\end{eqnarray}
Since $\Phi_1$ is the sum of the cases 1), 2) and 3), $\Phi_1=2(\lambda_2-\lambda_1)\pi\langle v, dx_n\rangle\Omega_3dx'$.

\begin{thm}
Let $M$ be a $4$-dimensional oriented
compact manifold with the boundary $\partial M$ and the metric
$g^M$ as above, $D_i~(i=1,2)$ be statistical de Rham Hodge Operators on $\widehat{M}$, then
\begin{eqnarray}
\widetilde{{\rm Wres}}[\pi^+D_2^{-1}\circ\pi^+D_1^{-1}]&=&32\pi^2\int_{M}
\bigg(-\frac{4}{3}s-4(\lambda^2_1+\lambda^2_2)|v|^2\bigg)d{\rm Vol_{M}}\nonumber\\
&&+\int_{\partial M}2(\lambda_2-\lambda_1)\pi\langle v, dx_n\rangle\Omega_3dx'.
\end{eqnarray}
where $s$ is the scalar curvature.
\end{thm}
Let $D=d+\delta+l(v),~D^*=d+\delta+\varepsilon(v^*)$.
\begin{cor}
When $\lambda_1=\lambda_2=1$, we get for a $4$-dimensional oriented
compact manifold $M$ with the boundary $\partial M$
\begin{eqnarray*}
\widetilde{{\rm Wres}}[\pi^+D^{-1}\circ\pi^+D^{-1}]&=&32\pi^2\int_{M}
\bigg(-\frac{4}{3}s-8|v|^2\bigg)d{\rm Vol_{M}}.
\end{eqnarray*}
where $s$ is the scalar curvature.
\end{cor}

On the other hand, we also prove the Kastler-Kalau-Walze type theorem for $4$-dimensional manifolds with boundary associated with $(D_i^*)^2~(i=1,2)$.
By (3.3) and (3.4), we will compute
\begin{equation}
\widetilde{{\rm Wres}}[\pi^+(D_1^*)^{-1}\circ\pi^+(D_2^*)^{-1}]=\int_M\int_{|\xi|=1}{\rm
trace}_{\wedge^*T^*M}[\sigma_{-4}((D_2^*D_1^*)^{-1})]\sigma(\xi)dx+\int_{\partial M}\Phi_2,
\end{equation}
where
\begin{eqnarray}
\Phi_2 &=&\int_{|\xi'|=1}\int^{+\infty}_{-\infty}\sum^{\infty}_{j, k=0}\sum\frac{(-i)^{|\alpha|+j+k+1}}{\alpha!(j+k+1)!}
\times {\rm trace}_{\wedge^*T^*M}[\partial^j_{x_n}\partial^\alpha_{\xi'}\partial^k_{\xi_n}\sigma^+_{r}
((D_1^*)^{-1})(x',0,\xi',\xi_n)
\nonumber\\
&&\times\partial^\alpha_{x'}\partial^{j+1}_{\xi_n}\partial^k_{x_n}\sigma_{l}((D_2^*)^{-1})(x',0,\xi',\xi_n)]d\xi_n\sigma(\xi')dx',
\end{eqnarray}
and the sum is taken over $r+l-k-j-|\alpha|=-3,~~r\leq -1,l\leq-1$.\\

Locally we can use Theorem 2.2 (2.43) to compute the interior of $\widetilde{{\rm Wres}}[\pi^+(D_1^*)^{-1}\circ\pi^+(D_2^*)^{-1}]$, we have
\begin{eqnarray}
&&\int_M\int_{|\xi|=1}{\rm
trace}_{\wedge^*T^*M}[\sigma_{-4}((D_2^*D_1^*)^{-1})]\sigma(\xi)dx=32\pi^2\int_{M}
\bigg(-\frac{4}{3}s-4(\lambda^2_1+\lambda^2_2)|v^*|^2\bigg)d{\rm Vol_{M}}
\end{eqnarray}

So we only need to compute $\int_{\partial M} \Phi_2$. From the remark above, now we can compute $\Phi_2$ (see formula (3.70) for the definition of $\Phi_2$). We use ${\rm tr}$ as shorthand of ${\rm trace}$. Since $n=4$, then ${\rm tr}_{\wedge^*T^*M}[{\rm \texttt{id}}]=16$, since the sum is taken over $
r+l-k-j-|\alpha|=-3,~~r\leq -1,l\leq-1,$ then we have the following five cases:

~\\
\noindent  {\bf case a)~I)}~$r=-1,~l=-1,~k=j=0,~|\alpha|=1$\\

\noindent By (3.70), we get
\begin{equation}
{\rm case~a)~I)}=-\int_{|\xi'|=1}\int^{+\infty}_{-\infty}\sum_{|\alpha|=1}
 {\rm tr}[\partial^\alpha_{\xi'}\pi^+_{\xi_n}\sigma_{-1}((D_1^*)^{-1})\times
 \partial^\alpha_{x'}\partial_{\xi_n}
 \sigma_{-1}((D_2^*)^{-1})](x_0)d\xi_n\sigma(\xi')dx'.
\end{equation}
\noindent  {\bf case a)~II)}~$r=-1,~l=-1,~k=|\alpha|=0,~j=1$\\

\noindent By (3.70), we get
\begin{equation}
{\rm case~ a)~II)}=-\frac{1}{2}\int_{|\xi'|=1}\int^{+\infty}_{-\infty} {\rm
tr} [\partial_{x_n}\pi^+_{\xi_n}\sigma_{-1}((D_1^*)^{-1})\times
\partial_{\xi_n}^2\sigma_{-1}((D_2^*)^{-1})](x_0)d\xi_n\sigma(\xi')dx'.
\end{equation}

\noindent  {\bf case a)~III)}~$r=-1,~l=-1,~j=|\alpha|=0,~k=1$\\

\noindent By (3.70), we get
\begin{equation}
{\rm case~ a)~III)}=-\frac{1}{2}\int_{|\xi'|=1}\int^{+\infty}_{-\infty}
{\rm tr} [\partial_{\xi_n}\pi^+_{\xi_n}\sigma_{-1}((D_1^*)^{-1})\times
\partial_{\xi_n}\partial_{x_n}\sigma_{-1}((D_2^*)^{-1})](x_0)d\xi_n\sigma(\xi')dx'.
\end{equation}
By Lemma 3.7, we have $\sigma_{-1}(D_i^{-1})=\sigma_{-1}((D_i^*)^{-1})~(i=1,2)$.
By (3.19)-(3.34), so {\bf case a)} vanishes.\\

\noindent  {\bf case b)}~$r=-2,~l=-1,~k=j=|\alpha|=0$\\

\noindent By (3.70), we get
\begin{eqnarray}
{\rm case~ b)}&=&-i\int_{|\xi'|=1}\int^{+\infty}_{-\infty}{\rm tr} [\pi^+_{\xi_n}\sigma_{-2}((D_1^*)^{-1})\times
\partial_{\xi_n}\sigma_{-1}((D_2^*)^{-1})](x_0)d\xi_n\sigma(\xi')dx'.
\end{eqnarray}
By Lemma 3.7, we have $\sigma_{-1}(D_i^{-1})=\sigma_{-1}((D_i^*)^{-1})~(i=1,2)$ and
\begin{eqnarray}
\sigma_{-2}((D_1^*)^{-1})(x_0)
=\frac{c(\xi)\sigma_{0}(D_1^*)(x_0)c(\xi)}{|\xi|^4}+\frac{c(\xi)}{|\xi|^6}c(dx_n)
[\partial_{x_n}[c(\xi')](x_0)|\xi|^2-c(\xi)h'(0)|\xi|^2_{\partial
M}],
\end{eqnarray}
where $\sigma_{0}(D_1^*)(x_0)=A(x_0)+B(x_0)+\lambda_1\varepsilon(v^*).$
Then
\begin{eqnarray}
&&\pi^+_{\xi_n}\sigma_{-2}((D_1^*)^{-1}(x_0))|_{|\xi'|=1}\nonumber\\
&=&\pi^+_{\xi_n}\Big[\frac{c(\xi)A(x_0)c(\xi)}{(1+\xi_n^2)^2}\Big]+\pi^+_{\xi_n}
\Big[\frac{c(\xi)(\lambda_1\varepsilon(v^*)(x_0))c(\xi)}{(1+\xi_n^2)^2}\Big]
+\pi^+_{\xi_n}\Big[\frac{c(\xi)B(x_0)c(\xi)
+c(\xi)c(dx_n)\partial_{x_n}[c(\xi')](x_0)}{(1+\xi_n^2)^2}\nonumber\\
&&-h'(0)\frac{c(\xi)c(dx_n)c(\xi)}{(1+\xi_n^{2})^3}\Big].
\end{eqnarray}
By (3.40)-(3.50), we have\\
\begin{eqnarray}
{\rm case~ b)}
&=&\frac{9}{2}\pi h'(0)\Omega_3dx'-i\int_{|\xi'|=1}\int^{+\infty}_{-\infty}{\rm trace} [\pi^+_{\xi_n}\Big[\frac{c(\xi)(\lambda_1\varepsilon(v^*)(x_0))c(\xi)}{(1+\xi_n^2)^2}\Big]\times\nonumber\\
&&
\partial_{\xi_n}\sigma_{-1}((D_2^*)^{-1})](x_0)d\xi_n\sigma(\xi')dx'.
\end{eqnarray}
Similar to (3.51), we have
\begin{eqnarray}
&&{\rm tr }[\pi^+_{\xi_n}\Big[\frac{c(\xi)(\lambda_1\varepsilon(v^*)(x_0))c(\xi)}{(1+\xi_n^2)^2}\Big]
\times\partial_{\xi_n}\sigma_{-1}(D_2^*)^{-1})(x_0)]|_{|\xi'|=1}\nonumber\\
&=&\frac{1}{2(1+\xi_n^2)^2}{\rm tr }[c(\xi')\lambda_1\varepsilon(v^*)(x_0)]
+\frac{i}{2(1+\xi_n^2)^2}{\rm tr }[c(dx_n)\lambda_1\varepsilon(v^*)(x_0)].
\end{eqnarray}
By the relation of the Clifford action and ${\rm tr}{AB}={\rm tr }{BA}$,  we have the equalities:\\
\begin{eqnarray}
{\rm tr}[c(dx_n)\varepsilon(v^*)]=-8\langle v^*,\frac{\partial}{\partial x_n}\rangle;
~~{\rm tr }[c(\xi')\varepsilon(v^*)]=-8\langle v^*,g(\xi',\cdot)\rangle;
\end{eqnarray}
By (3.79) and (3.80), we have
\begin{eqnarray}
-i\int_{|\xi'|=1}\int^{+\infty}_{-\infty}{\rm tr} [\pi^+_{\xi_n}
\Big[\frac{c(\xi)\lambda_1\varepsilon(v^*)c(\xi)}{(1+\xi_n^2)^2}\Big]\times
\partial_{\xi_n}\sigma_{-1}(D_2^{-1})](x_0)d\xi_n\sigma(\xi')dx'
=-2\lambda_1\pi\langle v^*,\frac{\partial}{\partial x_n}\rangle\Omega_3dx'
.
\end{eqnarray}
By (3.78) and (3.81), we have\\
\begin{eqnarray}
{\rm case~b)}=[\frac{9}{2}\pi h'(0)-2\lambda_1\pi\langle v^*,\frac{\partial}{\partial x_n}\rangle]\Omega_3dx'.
\end{eqnarray}

\noindent {\bf  case c)}~$r=-1,~l=-2,~k=j=|\alpha|=0$\\
By (3.70), we get
\begin{equation}
{\rm case~ c)}=-i\int_{|\xi'|=1}\int^{+\infty}_{-\infty}{\rm tr} [\pi^+_{\xi_n}\sigma_{-1}((D_1^*)^{-1})\times
\partial_{\xi_n}\sigma_{-2}((D_2^*)^{-1})](x_0)d\xi_n\sigma(\xi')dx'.
\end{equation}
By Lemma 3.7, we have $\sigma_{-1}(D_i^{-1})=\sigma_{-1}((D_i^*)^{-1})~(i=1,2)$.
Similar to (3.57), we have
\begin{eqnarray}
&&\partial_{\xi_n}\sigma_{-2}((D_2^*)^{-1})(x_0)|_{|\xi'|=1}\nonumber\\
&=&
\partial_{\xi_n}\bigg\{\frac{c(\xi)[A(x_0)+B(x_0)
+(\lambda_2\varepsilon(v^*)(x_{0}))]c(\xi)}{|\xi|^4}
+\frac{c(\xi)}{|\xi|^6}c(dx_n)[\partial_{x_n}[c(\xi')](x_0)|\xi|^2-c(\xi)h'(0)]\bigg\}\nonumber\\
&=&\partial_{\xi_n}\bigg\{\frac{c(\xi)A(x_0)]c(\xi)}{|\xi|^4}
+\frac{c(\xi)}{|\xi|^6}c(dx_n)[\partial_{x_n}[c(\xi')](x_0)|\xi|^2-c(\xi)h'(0)]\bigg\}
+\partial_{\xi_n}\frac{c(\xi)B(x_0)c(\xi)}{|\xi|^4}\nonumber\\
&&+\partial_{\xi_n}\frac{c(\xi)(\lambda_2\varepsilon(v^*)(x_{0}))c(\xi)}{|\xi|^4}.
\end{eqnarray}
By (3.58)-(3.65), we have
\begin{eqnarray}
{\rm case~ c)}&=&-\frac{9}{2}\pi h'(0)-i\int_{|\xi'|=1}\int^{+\infty}_{-\infty}{\rm tr}[\pi^+_{\xi_n}\sigma_{-1}((D_1^*)^{-1})\times\nonumber\\
&&\partial_{\xi_n}\frac{c(\xi)(\lambda_2\varepsilon(v^*))c(\xi)}
{|\xi|^4}](x_0)d\xi_n\sigma(\xi')dx'.
\end{eqnarray}
Similar to (3.66), we have
\begin{eqnarray}
&&-i\int_{|\xi'|=1}\int^{+\infty}_{-\infty}{\rm tr}[\pi^+_{\xi_n}\sigma_{-1}((D_1^*)^{-1})\times
\partial_{\xi_n}\frac{c(\xi)(\lambda_2\varepsilon(v^*))c(\xi)}
{|\xi|^4}](x_0)d\xi_n\sigma(\xi')dx'\nonumber\\
&=&-i\int_{|\xi'|=1}\int^{+\infty}_{-\infty}\frac{i}{(\xi-i)(\xi+i)^3}\Big[{\rm tr}[c(dx_n)(\lambda_2\varepsilon(v^*))(x_0)]+i{\rm tr}[c(\xi')(\lambda_2\varepsilon(v^*))(x_0)]\Big]d\xi_n\sigma(\xi')dx'\nonumber\\
&=&-\frac{\pi}{4}\Big[{\rm tr}[c(dx_n)(\lambda_2\varepsilon(v^*))]+i{\rm tr}[c(\xi')(\lambda_2\varepsilon(v^*))(x_0)]\Big]\Omega_3dx'\nonumber\\
&=&2\lambda_2\pi\langle v^*,\frac{\partial}{\partial x_n}\rangle\Omega_3dx'
.
\end{eqnarray}
So, we have
\begin{eqnarray}
{\rm case~ c)}&=&\Big[-\frac{9}{2}\pi h'(0)+2\lambda_2\pi\langle v^*,\frac{\partial}{\partial x_n}\rangle\Big]\Omega_3dx'.
\end{eqnarray}
Since $\Phi_2$ is the sum of the cases a), b) and c), so $\Phi_2=2(\lambda_2-\lambda_1)\pi\langle v^*,\frac{\partial}{\partial x_n}\rangle\Omega_3dx'.$\\

\begin{thm}
Let $M$ be a $4$-dimensional oriented
compact manifold with the boundary $\partial M$ and the metric
$g^M$ as above, $D_i$ and $D^*_i~(i=1,2)$ be statistical de Rham Hodge Operators on $\widehat{M}$, then
\begin{eqnarray}
\widetilde{{\rm Wres}}[\pi^+(D^*_1)^{-1}\circ\pi^+(D^*_2)^{-1}]&=&32\pi^2\int_{M}
\bigg(-\frac{4}{3}s-4(\lambda^2_1+\lambda^2_2)|v^*|^2\bigg)d{\rm Vol_{M}}\nonumber\\
&&+\int_{\partial M}2(\lambda_2-\lambda_1)\pi\langle v^*,\frac{\partial}{\partial x_n}\rangle\Omega_3dx'
.
\end{eqnarray}
where $s$ is the scalar curvature.
\end{thm}

\begin{cor}
When $\lambda_1=\lambda_2=1$, we get for a $4$-dimensional oriented
compact manifold $M$ with the boundary $\partial M$
\begin{eqnarray*}
\widetilde{{\rm Wres}}[\pi^+(D^*)^{-1}\circ\pi^+(D^*)^{-1}]&=&32\pi^2\int_{M}
\bigg(-\frac{4}{3}s-8|v^*|^2\bigg)d{\rm Vol_{M}}
.
\end{eqnarray*}
where $s$ is the scalar curvature.
\end{cor}

Next, we prove the Kastler-Kalau-Walze type theorem for $4$-dimensional manifolds with boundary associated with $D_2^*D_1$.
By (3.3) and (3.4), we will compute
\begin{equation}
\widetilde{{\rm Wres}}[\pi^+D_1^{-1}\circ\pi^+(D_2^*)^{-1}]=\int_M\int_{|\xi|=1}{\rm
trace}_{\wedge^*T^*M}[\sigma_{-4}((D_2^*D_1)^{-1})]\sigma(\xi)dx+\int_{\partial M}\Phi_3,
\end{equation}
where
\begin{eqnarray}
\Phi_3 &=&\int_{|\xi'|=1}\int^{+\infty}_{-\infty}\sum^{\infty}_{j, k=0}\sum\frac{(-i)^{|\alpha|+j+k+1}}{\alpha!(j+k+1)!}
\times {\rm trace}_{\wedge^*T^*M}[\partial^j_{x_n}\partial^\alpha_{\xi'}\partial^k_{\xi_n}\sigma^+_{r}
(D_1^{-1})(x',0,\xi',\xi_n)
\nonumber\\
&&\times\partial^\alpha_{x'}\partial^{j+1}_{\xi_n}\partial^k_{x_n}\sigma_{l}((D_2^*)^{-1})(x',0,\xi',\xi_n)]d\xi_n\sigma(\xi')dx',
\end{eqnarray}
and the sum is taken over $r+l-k-j-|\alpha|=-3,~~r\leq -1,l\leq-1$.\\

Locally we can use Theorem 2.2 (2.44) to compute the interior of $\widetilde{{\rm Wres}}[\pi^+D_1^{-1}\circ\pi^+(D_2^*)^{-1}]$, we have
\begin{eqnarray}
&&\int_M\int_{|\xi|=1}{\rm
trace}_{\wedge^*T^*M}[\sigma_{-4}((D_2^*D_1)^{-1})]\sigma(\xi)dx\nonumber\\
&=&32\pi^2\int_{M}
\bigg[-\frac{4}{3}s
-2(\lambda^2_1+\lambda^2_2-4\lambda_1\lambda_2)|v|^2+\frac{1}{2}{\rm tr}[\lambda_2\nabla^{TM}_{e_{j}}(\varepsilon(v^*))c(e_{j})\nonumber\\
&&-\lambda_1c(e_{j})\nabla^{TM}_{e_{j}}(l(v))]
\bigg]d{\rm Vol_{M}}.
\end{eqnarray}

So we only need to compute $\int_{\partial M} \Phi_3$. From the remark above, now we can compute $\Phi_3$ (see formula (3.89) for the definition of $\Phi_3$). We use ${\rm tr}$ as shorthand of ${\rm trace}$. Since $n=4$, then ${\rm tr}_{\wedge^*T^*M}[{\rm \texttt{id}}]=16$, since the sum is taken over $
r+l-k-j-|\alpha|=-3,~~r\leq -1,l\leq-1,$ then we have the following five cases:

~\\
\noindent  {\bf case a)~I)}~$r=-1,~l=-1,~k=j=0,~|\alpha|=1$\\

\noindent By (3.89), we get
\begin{equation}
{\rm case~a)~I)}=-\int_{|\xi'|=1}\int^{+\infty}_{-\infty}\sum_{|\alpha|=1}
 {\rm tr}[\partial^\alpha_{\xi'}\pi^+_{\xi_n}\sigma_{-1}(D_1^{-1})\times
 \partial^\alpha_{x'}\partial_{\xi_n}
 \sigma_{-1}((D_2^*)^{-1})](x_0)d\xi_n\sigma(\xi')dx'.
\end{equation}
\noindent  {\bf case a)~II)}~$r=-1,~l=-1,~k=|\alpha|=0,~j=1$\\

\noindent By (3.89), we get
\begin{equation}
{\rm case~ a)~II)}=-\frac{1}{2}\int_{|\xi'|=1}\int^{+\infty}_{-\infty} {\rm
tr} [\partial_{x_n}\pi^+_{\xi_n}\sigma_{-1}((D_1^{-1}))\times
\partial_{\xi_n}^2\sigma_{-1}((D_2^*)^{-1})](x_0)d\xi_n\sigma(\xi')dx'.
\end{equation}

\noindent  {\bf case a)~III)}~$r=-1,~l=-1,~j=|\alpha|=0,~k=1$\\

\noindent By (3.89), we get
\begin{equation}
{\rm case~ a)~III)}=-\frac{1}{2}\int_{|\xi'|=1}\int^{+\infty}_{-\infty}
{\rm tr} [\partial_{\xi_n}\pi^+_{\xi_n}\sigma_{-1}((D_1^{-1})\times
\partial_{\xi_n}\partial_{x_n}\sigma_{-1}((D_2^*)^{-1})](x_0)d\xi_n\sigma(\xi')dx'.
\end{equation}
By Lemma 3.7, we have $\sigma_{-1}(D_i^{-1})=\sigma_{-1}((D_i^*)^{-1})~(i=1,2)$.
By (3.19)-(3.34), so {\bf case a)} vanishes.\\

\noindent  {\bf case b)}~$r=-2,~l=-1,~k=j=|\alpha|=0$\\

\noindent By (3.89), we get
\begin{eqnarray}
{\rm case~ b)}&=&-i\int_{|\xi'|=1}\int^{+\infty}_{-\infty}{\rm tr} [\pi^+_{\xi_n}\sigma_{-2}(D_1^{-1})\times
\partial_{\xi_n}\sigma_{-1}((D_2^*)^{-1})](x_0)d\xi_n\sigma(\xi')dx'.
\end{eqnarray}
By Lemma 3.7, we have $\sigma_{-1}(D_i^{-1})=\sigma_{-1}((D_i^*)^{-1})~(i=1,2)$.
By (3.35)-(3.54), we have\\
\begin{eqnarray}
{\rm case~b)}=[\frac{9}{2}\pi h'(0)+2\lambda_1\pi\langle dx_n,v\rangle]\Omega_3dx'.
\end{eqnarray}

\noindent {\bf  case c)}~$r=-1,~l=-2,~k=j=|\alpha|=0$\\
By (3.70), we get
\begin{equation}
{\rm case~ c)}=-i\int_{|\xi'|=1}\int^{+\infty}_{-\infty}{\rm tr} [\pi^+_{\xi_n}\sigma_{-1}(D_1^{-1})\times
\partial_{\xi_n}\sigma_{-2}((D_2^*)^{-1})](x_0)d\xi_n\sigma(\xi')dx'.
\end{equation}
By Lemma 3.7, we have $\sigma_{-1}(D_i^{-1})=\sigma_{-1}((D_i^*)^{-1})~(i=1,2)$.
By (3.83)-(3.87), we have
\begin{eqnarray}
{\rm case~ c)}&=&\Big[-\frac{9}{2}\pi h'(0)+2\lambda_2\pi\langle v^*,\frac{\partial}{\partial x_n}\rangle\Big]\Omega_3dx'.
\end{eqnarray}
Since $\Phi_2$ is the sum of the cases a), b) and c), so
$$\Phi_3=2(\lambda_1+\lambda_2)\pi\langle dx_n,v\rangle\Omega_3dx'.$$

\begin{thm}
Let $M$ be a $4$-dimensional oriented
compact manifold with the boundary $\partial M$ and the metric
$g^M$ as above, $D_i^*~(i=1,2)$ be statistical de Rham Hodge Operators on $\widehat{M}$, then
\begin{eqnarray}
&&\widetilde{{\rm Wres}}[\pi^+D_1^{-1}\circ\pi^+(D_2^*)^{-1}]\nonumber\\
&=&32\pi^2\int_{M}
\bigg[-\frac{4}{3}s
-2(\lambda^2_1+\lambda^2_2-4\lambda_1\lambda_2)|v|^2+\frac{1}{2}{\rm tr}[\lambda_2\nabla^{TM}_{e_{j}}(\varepsilon(v^*))c(e_{j})\nonumber\\
&&-\lambda_1c(e_{j})\nabla^{TM}_{e_{j}}(l(v))]
\bigg]d{\rm Vol_{M}}
+\int_{\partial M}2(\lambda_1+\lambda_2)\pi\langle dx_n,v\rangle\Omega_3dx'.
\end{eqnarray}
where $s$ is the scalar curvature.
\end{thm}

\begin{cor}
When $\lambda_1=\lambda_2=1$, we get for a $4$-dimensional oriented
compact manifold $M$ with the boundary $\partial M$
\begin{eqnarray*}
&&\widetilde{{\rm Wres}}[\pi^+D^{-1}\circ\pi^+(D^*)^{-1}]\nonumber\\
&=&32\pi^2\int_{M}
\bigg[-\frac{4}{3}s
+4|v|^2+\frac{1}{2}{\rm tr}[\nabla^{TM}_{e_{j}}(\varepsilon(v^*))c(e_{j})
-c(e_{j})\nabla^{TM}_{e_{j}}(l(v))]
\bigg]d{\rm Vol_{M}}
\nonumber\\
&&+\int_{\partial M}4\pi\langle dx_n,v\rangle\Omega_3dx'.
\end{eqnarray*}
where $s$ is the scalar curvature.
\end{cor}
\section{A Kastler-Kalau-Walze type theorem for $6$-dimensional with boundary }
In this section, we prove the Kastler-Kalau-Walze type theorems for $6$-dimensional manifolds with boundary. An application of (2.1.4) in \cite{Wa5} shows that

\begin{equation}
\widetilde{{\rm Wres}}[\pi^+D_1^{-1}\circ\pi^+(D_2^{*}D_1
      D_2^{*})^{-1}]=\int_M\int_{|\xi|=1}{\rm
trace}_{\wedge ^*T^*M}[\sigma_{-4}((D_2^*D_1)^{-2})]\sigma(\xi)dx+\int_{\partial M}\Psi,
\end{equation}
where
\begin{eqnarray}
\Psi &=&\int_{|\xi'|=1}\int^{+\infty}_{-\infty}\sum^{\infty}_{j, k=0}\sum\frac{(-i)^{|\alpha|+j+k+1}}{\alpha!(j+k+1)!}
\times {\rm trace}_{\wedge ^*T^*M}[\partial^j_{x_n}\partial^\alpha_{\xi'}\partial^k_{\xi_n}\sigma^+_{r}(D_1^{-1})(x',0,\xi',\xi_n)
\nonumber\\
&&\times\partial^\alpha_{x'}\partial^{j+1}_{\xi_n}\partial^k_{x_n}\sigma_{l}
((D_2^{*}D_1
      D_2^{*})^{-1})(x',0,\xi',\xi_n)]d\xi_n\sigma(\xi')dx',
\end{eqnarray}
and the sum is taken over $r+\ell-k-j-|\alpha|-1=-6, \ r\leq-1, \ell\leq -3$.\\
Locally we can use Theorem 2.2 (2.44) to compute the interior term of (4.1), we have
\begin{eqnarray}
&&\int_M\int_{|\xi|=1}{\rm
trace}_{\wedge^*T^*M}[\sigma_{-4}((D_2^*D_1)^{-2})]\sigma(\xi)dx\nonumber\\
&=&128\pi^3\int_{M}
\bigg[-\frac{16}{3}s-8(\lambda^2_1+\lambda^2_2-8\lambda_1\lambda_2)|v|^2+\frac{1}{2}{\rm tr}[\lambda_2\nabla^{TM}_{e_{j}}(\varepsilon(v^*))c(e_{j})\nonumber\\
&&-\lambda_1c(e_{j})\nabla^{TM}_{e_{j}}(l(v))]
\bigg]d{\rm Vol_{M}}.
\end{eqnarray}

So we only need to compute $\int_{\partial M} \Psi$. Let us now turn to compute the specification of
$D_2^*D_1D_2^*$.
\begin{eqnarray}
D_2^*D_1D_2^*
&=&\sum^{n}_{i=1}c(e_{i})\langle e_{i},dx_{l}\rangle(-g^{ij}\partial_{l}\partial_{i}\partial_{j})
+\sum^{n}_{i=1}c(e_{i})\langle e_{i},dx_{l}\rangle \bigg\{-(\partial_{l}g^{ij})\partial_{i}\partial_{j}-g^{ij}\bigg(4(\sigma_{i}
+a_{i})\partial_{j}-2\Gamma^{k}_{ij}\partial_{k}\bigg)\partial_{l}\bigg\} \nonumber\\
&&+\sum^{n}_{i=1}c(e_{i})\langle e_{i},dx_{l}\rangle \bigg\{-2(\partial_{l}g^{ij})(\sigma_{i}+a_i)\partial_{j}+g^{ij}
(\partial_{l}\Gamma^{k}_{ij})\partial_{k}-2g^{ij}[(\partial_{l}\sigma_{i})
+(\partial_{l}a_i)]\partial_{j}
         +(\partial_{l}g^{ij})\Gamma^{k}_{ij}\partial_{k}\nonumber\\
         &&+\sum_{j,k}\Big[\partial_{l}\Big(c(e_{j})\lambda_2\varepsilon(v^*)
         +\lambda_1l(v)c(e_{j})\Big)\Big]\langle e_{j},dx^{k}\rangle\partial_{k}
         +\sum_{j,k}\Big(c(e_{j})\lambda_2\varepsilon(v^*)
         +\lambda_1l(v)c(e_{j})\Big)\Big[\partial_{l}\langle e_{j},dx^{k}\rangle\Big]\partial_{k} \bigg\}\nonumber\\
         &&+\Big[(\sigma_{i}+a_{i})+\lambda_2\varepsilon(v^*)\Big](-g^{ij}\partial_{i}\partial_{j})
         +\sum^{n}_{i=1}c(e_{i})\langle e_{i},dx_{l}\rangle \bigg\{2\sum_{j,k}\Big(c(e_{j})\lambda_2\varepsilon(v^*)+\lambda_1l(v)c(e_{j})\Big)\nonumber\\
         &&\times\langle \widetilde{e_{i}},dx_{k}\rangle\bigg\}\partial_{l}\partial_{k}
         +\Big[(\sigma_{i}+a_{i})+\lambda_2\varepsilon(v^*)\Big]
         \bigg\{-\sum_{i,j}g^{i,j}\Big[2\sigma_{i}\partial_{j}+2a_{i}\partial_{j}
         -\Gamma_{i,j}^{k}\partial_{k}+(\partial_{i}\sigma_{j})+\frac{1}{4}s\nonumber\\
         &&+(\partial_{i}a_{j})+\sigma_{i}\sigma_{j}+\sigma_{i}a_{j}+a_{i}\sigma_{j}+a_{i}a_{j} -\Gamma_{i,j}^{k}\sigma_{k}-\Gamma_{i,j}^{k}a_{k}\Big]+\sum_{i,j}g^{i,j}\Big(c(e_{j})
         \lambda_2\varepsilon(v^*)
         +\lambda_1l(v)c(e_{j})\Big)\partial_{j}\nonumber\\
         &&+\sum_{i,j}g^{i,j}\Big[\lambda_1l(v)c(\partial_{i})\sigma_{i}+\lambda_1l(v)c(\partial_{i})a_{i}
         +c(\partial_{i})\lambda_2\partial_{i}(\varepsilon(v^*))+c(\partial_{i})\sigma_{i}\lambda_2\varepsilon(v^*)
         +c(\partial_{i})a_{i}\lambda_2\varepsilon(v^*)\Big]\nonumber\\
         &&+\lambda_1\lambda_2l(v)\varepsilon(v^*)-\frac{1}{8}\sum_{ijkl}R_{ijkl}\bar{c}(e_i)\bar{c}(e_j)
         c(e_k)c(e_l)
         \bigg\}.
\end{eqnarray}
Then, we obtain
\begin{lem} The following identities hold:
\begin{eqnarray}
\sigma_2(D_2^*D_1D_2^*)&=&\sum_{i,j,l}c(dx_{l})\partial_{l}(g^{i,j})\xi_{i}\xi_{j} +c(\xi)(4\sigma^k+4a^k-2\Gamma^k)\xi_{k}-2[\lambda_1c(\xi)l(v)c(\xi)-|\xi|^2\lambda_2\varepsilon(v^*)]\nonumber\\
&&+\frac{1}{4}|\xi|^2\sum_{s,t,l}\omega_{s,t}
(e_l)[c(e_l)\bar{c}(e_s)\bar{c}(e_t)
-c(e_l)c(e_s)c(e_t)]
+|\xi|^2\lambda_2\varepsilon(v^*);\nonumber\\
\sigma_{3}(D_2^*D_1D_2^*)
&=&ic(\xi)|\xi|^{2}.
\end{eqnarray}
\end{lem}

Write
\begin{eqnarray}
\sigma(D_2^*D_1D_2^*)&=&p_3+p_2+p_1+p_0;
~\sigma((D_2^*D_1D_2^*)^{-1})=\sum^{\infty}_{j=3}q_{-j}.
\end{eqnarray}

By the composition formula of pseudodifferential operators, we have

\begin{eqnarray}
1&=&\sigma((D_2^*D_1D_2^*)\circ (D_2^*D_1D_2^*)^{-1})\nonumber\\
&=&
\sum_{\alpha}\frac{1}{\alpha!}\partial^{\alpha}_{\xi}
[\sigma(D_2^*D_1D_2^*)]D^{\alpha}_{x}
[(D_2^*D_1D_2^*)^{-1}] \nonumber\\
&=&(p_3+p_2+p_1+p_0)(q_{-3}+q_{-4}+q_{-5}+\cdots) \nonumber\\
&+&\sum_j(\partial_{\xi_j}p_3+\partial_{\xi_j}p_2+\partial_{\xi_j}p_1+\partial_{\xi_j}p_0)
(D_{x_j}q_{-3}+D_{x_j}q_{-4}+D_{x_j}q_{-5}+\cdots) \nonumber\\
&=&p_3q_{-3}+(p_3q_{-4}+p_2q_{-3}+\sum_j\partial_{\xi_j}p_3D_{x_j}q_{-3})+\cdots,
\end{eqnarray}
by (4.7), we have

\begin{equation}
q_{-3}=p_3^{-1};~q_{-4}=-p_3^{-1}[p_2p_3^{-1}+\sum_j\partial_{\xi_j}p_3D_{x_j}(p_3^{-1})].
\end{equation}
By Lemma 4.1, we have some symbols of operators.
\begin{lem} The following identities hold:
\begin{eqnarray}
\sigma_{-3}((D_2^*D_1D_2^*)^{-1})&=&\frac{ic(\xi)}{|\xi|^{4}};\nonumber\\
\sigma_{-4}((D_2^*D_1D_2^*)^{-1})&=&
\frac{c(\xi)\sigma_{2}(D_2^*D_1D_2^*)c(\xi)}{|\xi|^8}
+\frac{ic(\xi)}{|\xi|^8}\Big(|\xi|^4c(dx_n)\partial_{x_n}c(\xi')
-2h'(0)c(dx_n)c(\xi)\nonumber\\
&&+2\xi_{n}c(\xi)\partial_{x_n}c(\xi')+4\xi_{n}h'(0)\Big).
\end{eqnarray}
\end{lem}

In the normal coordinate, $g^{ij}(x_{0})=\delta^{j}_{i}$ and $\partial_{x_{j}}(g^{\alpha\beta})(x_{0})=0$, if $j<n$; $\partial_{x_{j}}(g^{\alpha\beta})(x_{0})=h'(0)\delta^{\alpha}_{\beta}$, if $j=n$.
So by Lemma A.2 in \cite{Wa3}, we have $\Gamma^{n}(x_{0})=\frac{5}{2}h'(0)$ and $\Gamma^{k}(x_{0})=0$ for $k<n$. By the definition of $\delta^{k}$ and Lemma 2.3 in \cite{Wa3}, we have $\delta^{n}(x_{0})=0$ and $\delta^{k}=\frac{1}{4}h'(0)c(e_{k})c(e_{n})$ for $k<n$. By Lemma 4.2, we obtain

\begin{eqnarray}
\sigma_{-4}((D_2^{*}D_1D_2^{*})^{-1})(x_{0})|_{|\xi'|=1}&=&
\frac{c(\xi)\sigma_{2}((D_2^{*}D_1D_2^{*})^{-1})
(x_{0})|_{|\xi'|=1}c(\xi)}{|\xi|^8}
-\frac{c(\xi)}{|\xi|^4}\sum_j\partial_{\xi_j}\big(c(\xi)|\xi|^2\big)
D_{x_j}\big(\frac{ic(\xi)}{|\xi|^4}\big)\nonumber\\
&=&\frac{1}{|\xi|^8}c(\xi)\Big(\frac{1}{2}h'(0)c(\xi)\sum_{k<n}\xi_k
c(e_k)c(e_n)-\frac{1}{2}h'(0)c(\xi)\sum_{k<n}\xi_k
\bar{c}(e_k)\bar{c}(e_n)\nonumber\\
&&-\frac{5}{2}h'(0)\xi_nc(\xi)-\frac{1}{4}h'(0)|\xi|^2c(dx_n)
-2[c(\xi)\lambda_1l(v)c(\xi)-|\xi|^2\lambda_2\varepsilon(v^*)]\nonumber\\
&&+|\xi|^2\lambda_2\varepsilon(v^*)\Big)c(\xi)
+\frac{ic(\xi)}{|\xi|^8}\Big(|\xi|^4c(dx_n)\partial_{x_n}c(\xi')
-2h'(0)c(dx_n)c(\xi)\nonumber\\
&&+2\xi_{n}c(\xi)\partial_{x_n}c(\xi')+4\xi_{n}h'(0)\Big).
\end{eqnarray}

From the remark above, now we can compute $\Psi$ (see formula (4.2) for the definition of $\Psi$). We use ${\rm tr}$ as shorthand of ${\rm trace}$. Since $n=6$, ${\rm tr}_{\wedge ^*T^*M}[\texttt{id}]=64$.
Since the sum is taken over $r+\ell-k-j-|\alpha|-1=-6, \ r\leq-1, \ell\leq -3$, we have the $\int_{\partial_{M}}\Psi$
is the sum of the following five cases:

~\\
\noindent  {\bf case (a)~(I)}~$r=-1, l=-3, j=k=0, |\alpha|=1$.\\
By (4.2), we get
 \begin{equation}
{\rm case~(a)~(I)}=-\int_{|\xi'|=1}\int^{+\infty}_{-\infty}\sum_{|\alpha|=1}{\rm tr}
\Big[\partial^{\alpha}_{\xi'}\pi^{+}_{\xi_{n}}\sigma_{-1}(D_1^{-1})
      \times\partial^{\alpha}_{x'}\partial_{\xi_{n}}\sigma_{-3}((D_2^{*}D_1
      D_2^{*})^{-1})\Big](x_0)d\xi_n\sigma(\xi')dx'.
\end{equation}
By Lemma 4.2, for $i<n$, we have
 \begin{equation}
 \partial_{x_{i}}\sigma_{-3}((D_2^{*}D_1D_2^{*})^{-1})(x_0)=
      \partial_{x_{i}}\Big[\frac{ic(\xi)}{|\xi|^{4}}\Big](x_{0})
      =i\partial_{x_{i}}[c(\xi)]|\xi|^{-4}(x_{0})-2ic(\xi)\partial_{x_{i}}[|\xi|^{2}]|\xi|^{-6}(x_{0})=0.
\end{equation}
 so {\rm case~(a)~(I)} vanishes.
~\\

\noindent  {\bf case (a)~(II)}~$r=-1, l=-3, |\alpha|=k=0, j=1$.\\
By (4.2), we have
  \begin{equation}
{\rm case~(a)~(II)}=-\frac{1}{2}\int_{|\xi'|=1}\int^{+\infty}_{-\infty} {\rm
tr} \Big[\partial_{x_{n}}\pi^{+}_{\xi_{n}}\sigma_{-1}(D_1^{-1})
      \times\partial^{2}_{\xi_{n}}\sigma_{-3}((D_2^{*}D_1D_2^{*})^{-1})\Big](x_0)d\xi_n\sigma(\xi')dx'.
\end{equation}
By Lemma 4.2 and direct calculations, we have\\
\begin{equation}
\partial^{2}_{\xi_{n}}\sigma_{-3}((D_2^{*}D_1D_2^{*})^{-1})=i\bigg[\frac{(20\xi^{2}_{n}-4)c(\xi')+
12(\xi^{3}_{n}-\xi_{n})c(dx_{n})}{(1+\xi_{n}^{2})^{4}}\bigg].
\end{equation}
Since $n=6$, ${\rm tr}[-\texttt{id}]=-64$. By the relation of the Clifford action and ${\rm tr}AB={\rm tr}BA$,  then
\begin{eqnarray}
&&{\rm tr}[c(\xi')c(dx_{n})]=0; \ {\rm tr}[c(dx_{n})^{2}]=-64;\
{\rm tr}[c(\xi')^{2}](x_{0})|_{|\xi'|=1}=-64;\nonumber\\
&&{\rm tr}[\partial_{x_{n}}[c(\xi')]c(\texttt{d}x_{n})]=0; \
{\rm tr}[\partial_{x_{n}}c(\xi')c(\xi')](x_{0})|_{|\xi'|=1}=-32h'(0).
\end{eqnarray}
By (3.31), (4.14) and (4.15), we get
\begin{equation}
{\rm
tr} \Big[\partial_{x_{n}}\pi^{+}_{\xi_{n}}\sigma_{-1}(D_1^{-1})
      \times\partial^{2}_{\xi_{n}}\sigma_{-3}((D_2^{*}D_1D_2^{*})^{-1})\Big](x_0)
=64 h'(0)\frac{-1-3\xi_{n}i+5\xi^{2}_{n}+3i\xi^{3}_{n}}{(\xi_{n}-i)^{6}(\xi_{n}+i)^{4}}.
\end{equation}
Then we obtain

\begin{eqnarray}
{\rm case~(a)~(II)}&=&-\frac{1}{2}\int_{|\xi'|=1}\int^{+\infty}_{-\infty} h'(0)\frac{-8-24\xi_{n}i+40\xi^{2}_{n}+24i\xi^{3}_{n}}{(\xi_{n}-i)^{6}(\xi_{n}+i)^{4}}d\xi_n\sigma(\xi')dx'\nonumber\\
     &=&-\frac{15}{2}\pi h'(0)\Omega_{4}dx',
\end{eqnarray}
where ${\rm \Omega_{4}}$ is the canonical volume of $S^{4}.$\\

\noindent  {\bf case (a)~(III)}~$r=-1,l=-3,|\alpha|=j=0,k=1$.\\
By (4.2), we have
 \begin{equation}
{\rm case~ (a)~(III)}=-\frac{1}{2}\int_{|\xi'|=1}\int^{+\infty}_{-\infty}{\rm tr} \Big[\partial_{\xi_{n}}\pi^{+}_{\xi_{n}}\sigma_{-1}(D_1^{-1})
      \times\partial_{\xi_{n}}\partial_{x_{n}}\sigma_{-3}((D_2^{*}D_1D_2^{*})^{-1})\Big](x_0)d\xi_n\sigma(\xi')dx'.
\end{equation}
By Lemma 4.2 and direct calculations, we have\\
\begin{equation}
\partial_{\xi_{n}}\partial_{x_{n}}\sigma_{-3}((D_2^{*}D_1D_2^{*})^{-1})=-\frac{4 i\xi_{n}\partial_{x_{n}}c(\xi')(x_{0})}{(1+\xi_{n}^{2})^{3}}
      +i\frac{12h'(0)\xi_{n}c(\xi')}{(1+\xi_{n}^{2})^{4}}
      -i\frac{(2-10\xi^{2}_{n})h'(0)c(dx_{n})}{(1+\xi_{n}^{2})^{4}}.
\end{equation}
Combining (3.31) and (4.19), we have
\begin{equation}
{\rm tr} \Big[\partial_{\xi_{n}}\pi^{+}_{\xi_{n}}\sigma_{-1}(D_1^{-1})
      \times\partial_{\xi_{n}}\partial_{x_{n}}\sigma_{-3}((D_2^{*}D_1D_2^{*})^{-1})\Big](x_{0})|_{|\xi'|=1}
=8h'(0)\frac{8i-32\xi_{n}-8i\xi^{2}_{n}}{(\xi_{n}-i)^{5}(\xi+i)^{4}}.
\end{equation}
Then
\begin{eqnarray}
{\rm case~(a)~III)}&=&-\frac{1}{2}\int_{|\xi'|=1}\int^{+\infty}_{-\infty} 8h'(0)\frac{8i-32\xi_{n}-8i\xi^{2}_{n}}{(\xi_{n}-i)^{5}(\xi+i)^{4}}d\xi_n\sigma(\xi')dx'\nonumber\\
      &=&\frac{25}{2}\pi h'(0)\Omega_{4}dx'.
\end{eqnarray}

\noindent  {\bf case (b)}~$r=-1,l=-4,|\alpha|=j=k=0$.\\
By (4.2), we have
 \begin{eqnarray}
{\rm case~ (b)}&=&-i\int_{|\xi'|=1}\int^{+\infty}_{-\infty}{\rm tr} \Big[\pi^{+}_{\xi_{n}}\sigma_{-1}(D_1^{-1})
      \times\partial_{\xi_{n}}\sigma_{-4}((D_2^{*}D_1
     D_2^{*})^{-1})\Big](x_0)d\xi_n\sigma(\xi')dx'\nonumber\\
&=&i\int_{|\xi'|=1}\int^{+\infty}_{-\infty}{\rm tr} [\partial_{\xi_n}\pi^+_{\xi_n}\sigma_{-1}(D_1^{-1})\times
\sigma_{-4}((D_2^{*}D_1
      D_2^{*})^{-1})](x_0)d\xi_n\sigma(\xi')dx'.
\end{eqnarray}

By (3.31) and (4.23), we have
\begin{eqnarray}
&&{\rm tr} [\partial_{\xi_n}\pi^+_{\xi_n}\sigma_{-1}(D_1^{-1})\times
\sigma_{-4}(D_2^{*}D_1D_2^{*})^{-1}](x_0)|_{|\xi'|=1} \nonumber\\
&&=\frac{1}{2(\xi_{n}-i)^{2}(1+\xi_{n}^{2})^{4}}\big(\frac{3}{4}i+2+(3+4i)\xi_{n}+(-6+2i)\xi_{n}^{2}+3\xi_{n}^{3}+\frac{9i}{4}\xi_{n}^{4}\big)h'(0){\rm tr}
[id]\nonumber\\
&&+\frac{1}{2(\xi_{n}-i)^{2}(1+\xi_{n}^{2})^{4}}\big(-1-3i\xi_{n}-2\xi_{n}^{2}-4i\xi_{n}^{3}-\xi_{n}^{4}-i\xi_{n}^{5}\big){\rm tr[c(\xi')\partial_{x_n}c(\xi')]}\nonumber\\
&&-\frac{1}{2(\xi_{n}-i)^{2}(1+\xi_{n}^{2})^{4}}\big(\frac{1}{2}i+\frac{1}{2}\xi_{n}+\frac{1}{2}\xi_{n}^{2}+\frac{1}{2}\xi_{n}^{3}\big){\rm tr}
[c(\xi')\bar{c}(\xi')c(dx_n)\bar{c}(dx_n)]\nonumber\\
&&+{\rm tr} \Big[\pi^+_{\xi_n}\sigma_{-1}(D_1^{-1})\times\partial_{\xi_n}
\Big(\frac{3c(\xi)\lambda_2\varepsilon(v^*)c(\xi)}{|\xi|^6}
-\frac{2\lambda_1l(v)}{|\xi|^4}\Big)\Big](x_0)|_{|\xi'|=1}
\end{eqnarray}
By direct calculation, we have
\begin{eqnarray}
&&{\rm tr} \Big[\pi^+_{\xi_n}\sigma_{-1}(D_1^{-1})\times\partial_{\xi_n}
\Big(\frac{3c(\xi)\lambda_2\varepsilon(v^*)c(\xi)}{|\xi|^6}
-\frac{2\lambda_1l(v)}{|\xi|^4}\Big)\Big](x_0)|_{|\xi'|=1}\nonumber\\
&=&\frac{3(4i\xi_n+2)i}{2(\xi_{n}+i)(1+\xi^2_{n})^{3}}{\rm tr}\big[\lambda_2\varepsilon(v^*)c(\xi')\big]+\frac{3(4i\xi_n+2)}{2(\xi_{n}+i)(1+\xi^2_{n})^{3}}{\rm tr}\big[\lambda_2\varepsilon(v^*)c(dx_n)\big]\nonumber\\
&&+\frac{4\xi_n}{2(\xi_{n}-i)(1+\xi^2_{n})^{3}}{\rm tr}\big[\lambda_1l(v)c(\xi')\big]+\frac{4\xi_ni}{2(\xi_{n}-i)(1+\xi^2_{n})^{3}}{\rm tr}\big[\lambda_1l(v)c(dx_n)\big].
\end{eqnarray}
By the relation of the Clifford action and ${\rm tr}{AB}={\rm tr }{BA}$, then we have the following equalities:
\begin{eqnarray}
&&{\rm tr}[c(dx_n)l(v)]=32\langle dx_n,v\rangle;
~~{\rm tr }[c(\xi')l(v)]=32\langle \xi',v\rangle;\nonumber\\
&&{\rm tr}[c(dx_n)\varepsilon(v^*)]=-32\langle v^*,\frac{\partial}{\partial x_n}\rangle;
~~{\rm tr }[c(\xi')\varepsilon(v^*)]=-32\langle v^*,g(\xi',\cdot)\rangle;\nonumber\\
&&{\rm tr}[c(e_i)\bar{c}(e_i)c(e_n)\bar{c}(e_n)]=0~~(i<n).
\end{eqnarray}
So
\begin{eqnarray}
{\rm tr}
[c(\xi')\bar{c}(\xi')c(dx_n)\bar{c}(dx_n)]&=&
\sum_{i<n,j<n}{\rm tr}[\xi_{i}\xi_{j}c(e_i)\bar{c}
(e_j)c(dx_n)\bar{c}(dx_n)]=0.
\end{eqnarray}
By (4.25), then we have
\begin{eqnarray}
{\rm case~ (b)}&=&
 ih'(0)\int_{|\xi'|=1}\int^{+\infty}_{-\infty}64\times\frac{\frac{3}{4}i+2+(3+4i)\xi_{n}+(-6+2i)\xi_{n}^{2}+3\xi_{n}^{3}+\frac{9i}{4}\xi_{n}^{4}}{2(\xi_n-i)^5(\xi_n+i)^4}d\xi_n\sigma(\xi')dx'\nonumber\\ &+&ih'(0)\int_{|\xi'|=1}\int^{+\infty}_{-\infty}32\times\frac{1+3i\xi_{n}
 +2\xi_{n}^{2}+4i\xi_{n}^{3}+\xi_{n}^{4}+i\xi_{n}^{5}}{2(\xi_{n}-i)^{2}
 (1+\xi_{n}^{2})^{4}}d\xi_n\sigma(\xi')dx'\nonumber\\
 &+&-i\int_{|\xi'|=1}\int^{+\infty}_{-\infty}{\rm tr} \Big[\pi^+_{\xi_n}\sigma_{-1}(D_1^{-1})\times\partial_{\xi_n}
\Big(\frac{3c(\xi)\lambda_2\varepsilon(v^*)c(\xi)}{|\xi|^6}
-\frac{2\lambda_1l(v)}{|\xi|^4}\Big)\Big](x_0)|_{|\xi'|=1}d\xi_n\sigma(\xi')dx'\nonumber\\
&=&(-\frac{41}{8}i-\frac{195}{8})\pi h'(0)\Omega_4dx'
+(4\lambda_1+18\lambda_2)\pi\langle dx_n,v\rangle\Omega_4dx'.
\end{eqnarray}

\noindent {\bf  case (c)}~$r=-2,l=-3,|\alpha|=j=k=0$.\\
By (4.2), we have

\begin{equation}
{\rm case~ (c)}=-i\int_{|\xi'|=1}\int^{+\infty}_{-\infty}{\rm tr} \Big[\pi^{+}_{\xi_{n}}\sigma_{-2}(D_1^{-1})
      \times\partial_{\xi_{n}}\sigma_{-3}((D_2^{*}D_1D_2^{*})^{-1})\Big](x_0)d\xi_n\sigma(\xi')dx'.
\end{equation}
By (3.36) and (3,37), we have
\begin{eqnarray}
&&\pi^+_{\xi_n}\sigma_{-2}(D_1^{-1}(x_0))|_{|\xi'|=1}\nonumber\\
&=&\pi^+_{\xi_n}\Big[\frac{c(\xi)A(x_0)c(\xi)}{(1+\xi_n^2)^2}\Big]+\pi^+_{\xi_n}
\Big[\frac{c(\xi)(\lambda_1l(v)(x_0))c(\xi)}{(1+\xi_n^2)^2}\Big]
+\pi^+_{\xi_n}\Big[\frac{c(\xi)B(x_0)c(\xi)
+c(\xi)c(dx_n)\partial_{x_n}[c(\xi')](x_0)}{(1+\xi_n^2)^2}\nonumber\\
&&-h'(0)\frac{c(\xi)c(dx_n)c(\xi)}{(1+\xi_n^{2})^3}\Big].
\end{eqnarray}
By (3.40), we have
\begin{eqnarray}
&&\pi^+_{\xi_n}\Big[\frac{c(\xi)A(x_0)c(\xi)}{(1+\xi_n^2)^2}\Big]\nonumber\\
&=&-\frac{c(\xi')A(x_0)c(\xi')(2+i\xi_{n})}{4(\xi_{n}-i)^{2}}
+\frac{ic(\xi')A(x_0)c(dx_{n})}{4(\xi_{n}-i)^{2}}
+\frac{ic(dx_{n})A(x_0)c(\xi')}{4(\xi_{n}-i)^{2}}
\nonumber\\
&&+\frac{-i\xi_{n}c(dx_{n})A(x_0)c(dx_{n})}{4(\xi_{n}-i)^{2}}.
\end{eqnarray}
By (3.45)-(3,47), we have
\begin{eqnarray}
\pi^+_{\xi_n}\Big[\frac{c(\xi)B(x_0)c(\xi)+c(\xi)c(dx_n)
\partial_{x_n}[c(\xi')](x_0)}{(1+\xi_n^2)^2}\Big]
-h'(0)\pi^+_{\xi_n}\Big[\frac{c(\xi)c(dx_n)c(\xi)}{(1+\xi_n)^3}\Big]:= P_1-P_2,
\end{eqnarray}
where
\begin{eqnarray}
P_1&=&\frac{-1}{4(\xi_n-i)^2}[(2+i\xi_n)c(\xi')b_0^{2}(x_0)c(\xi')+i\xi_nc(dx_n)b_0^{2}(x_0)c(dx_n)\nonumber\\
&&+(2+i\xi_n)c(\xi')c(dx_n)\partial_{x_n}c(\xi')+ic(dx_n)b_0^{2}(x_0)c(\xi')
+ic(\xi')b_0^{2}(x_0)c(dx_n)-i\partial_{x_n}c(\xi')]
\end{eqnarray}
and
\begin{eqnarray}
P_2&=&\frac{h'(0)}{2}\left[\frac{c(dx_n)}{4i(\xi_n-i)}+\frac{c(dx_n)-ic(\xi')}{8(\xi_n-i)^2}
+\frac{3\xi_n-7i}{8(\xi_n-i)^3}[ic(\xi')-c(dx_n)]\right].
\end{eqnarray}
Similar to (4.30), we have
\begin{eqnarray}
&&\pi^+_{\xi_n}\Big[\frac{c(\xi)(l(v)(x_0))c(\xi)}{(1+\xi_n^2)^2}\Big]\nonumber\\
&=&-\frac{c(\xi')(l(v)(x_0))c(\xi')(2+i\xi_{n})}{4(\xi_{n}-i)^{2}}
+\frac{ic(\xi')(l(v)(x_0))c(dx_{n})}{4(\xi_{n}-i)^{2}}
+\frac{ic(dx_{n})(l(v)(x_0))c(\xi')}{4(\xi_{n}-i)^{2}}
\nonumber\\
&&+\frac{-i\xi_{n}c(dx_{n})(l(v)(x_0))c(dx_{n})}{4(\xi_{n}-i)^{2}}.
\end{eqnarray}
On the other hand,
\begin{equation}
\partial_{\xi_{n}}\sigma_{-3}((D_1^{*}D_2D_1^{*})^{-1})=\frac{-4 i \xi_{n}c(\xi')}{(1+\xi_{n}^{2})^{3}}+\frac{i(1-3\xi_{n}^{2})c(\texttt{d}x_{n})}
{(1+\xi_{n}^{2})^{3}}.
\end{equation}
By the relation of the Clifford action and ${\rm tr}{AB}={\rm tr }{BA}$, then we have equalities:
\begin{eqnarray}
{\rm tr}[Ac(dx_n)]=0;~~{\rm tr}[\bar {c}(\xi')\bar {c}(dx_n)]=0.
\end{eqnarray}

Then we have
\begin{eqnarray}
{\rm tr }\bigg[\pi^+_{\xi_n}\Big(\frac{c(\xi)A(x_0)c(\xi)}{(1+\xi_n^2)^2}\Big)\times
\partial_{\xi_n}\sigma_{-3}((D_v^{*}D_vD_v^{*})^{-1})(x_0)\bigg]\bigg|_{|\xi'|=1}=\frac{2-8i\xi_n-6\xi_n^2}{4(\xi_n-i)^{2}(1+\xi_n^2)^{3}}{\rm tr }[A(x_0)c(\xi')],
\end{eqnarray}
We note that $i<n,~\int_{|\xi'|=1}\{\xi_{i_{1}}\xi_{i_{2}}\cdots\xi_{i_{2d+1}}\}\sigma(\xi')=0$,
so ${\rm tr }[A(x_0)c(\xi')]$ has no contribution for computing {\rm case~c)}.

By (4.32) and (4.35), we have
\begin{eqnarray}
&&{\rm tr }[P_1\times\partial_{\xi_n}\sigma_{-3}((D_1^{*}D_2D_1^{*})^{-1})(x_0)]|_{|\xi'|=1}\nonumber\\
&=&{\rm tr }\Big\{ \frac{1}{4(\xi_n-i)^2}\Big[\frac{5}{2}h'(0)c(dx_n)-\frac{5i}{2}h'(0)c(\xi')
  -(2+i\xi_n)c(\xi')c(dx_n)\partial_{\xi_n}c(\xi')+i\partial_{\xi_n}c(\xi')\Big]\nonumber\\
&&\times \frac{-4i\xi_nc(\xi')+(i-3i\xi_n^{2})c(dx_n)}{(1+\xi_n^{2})^3}\Big\} \nonumber\\
&=&8h'(0)\frac{3+12i\xi_n+3\xi_n^{2}}{(\xi_n-i)^4(\xi_n+i)^3}.
\end{eqnarray}
By (4.33) and (4.35), we have
\begin{eqnarray}
&&{\rm tr }[P_2\times\partial_{\xi_n}\sigma_{-3}((D_2^{*}D_1D_2^{*})^{-1})(x_0)]|_{|\xi'|=1}\nonumber\\
&=&{\rm tr }\Big\{ \frac{h'(0)}{2}\Big[\frac{c(dx_n)}{4i(\xi_n-i)}+\frac{c(dx_n)-ic(\xi')}{8(\xi_n-i)^2}
+\frac{3\xi_n-7i}{8(\xi_n-i)^3}[ic(\xi')-c(dx_n)]\Big] \nonumber\\
&&\times\frac{-4i\xi_nc(\xi')+(i-3i\xi_n^{2})c(dx_n)}{(1+\xi_n^{2})^3}\Big\} \nonumber\\
&=&8h'(0)\frac{4i-11\xi_n-6i\xi_n^{2}+3\xi_n^{3}}{(\xi_n-i)^5(\xi_n+i)^3}.
\end{eqnarray}
By (4.34) and (4.35), we have
\begin{eqnarray}
&&{\rm tr }\bigg[\pi^+_{\xi_n}\Big(\frac{c(\xi)[\lambda_1l(v)(x_0)]
c(\xi)}{(1+\xi_n^2)^2}\Big)\times
\partial_{\xi_n}\sigma_{-3}((D_2^{*}D_1D_2^{*})^{-1})
(x_0)\bigg]\bigg|_{|\xi'|=1}\nonumber\\
&=&\frac{-2-8i\xi_n+6\xi_n^2}{4(\xi_n-i)^{2}(1+\xi_n^2)^{3}}{\rm tr }[\lambda_1l(v)(x_0)c(\xi')]
+\frac{-2i+8\xi_n+6i\xi_n^2}{4(\xi_n-i)^{2}(1+\xi_n^2)^{3}}{\rm tr }[\lambda_1c(dx_n)l(v)(x_0)].
\end{eqnarray}
By (4.25), we have
\begin{eqnarray}
{\rm case~(c)}&=&
 -i h'(0)\int_{|\xi'|=1}\int^{+\infty}_{-\infty}
 8\times\frac{-7i+26\xi_n+15i\xi_n^{2}}{(\xi_n-i)^5(\xi_n+i)^3}d\xi_n\sigma(\xi')dx' \nonumber\\
 &&-i\int_{|\xi'|=1}\int^{+\infty}_{-\infty}
 {\rm tr }\bigg[\pi^+_{\xi_n}\Big(\frac{c(\xi)[\lambda_1l(v)(x_0)]
c(\xi)}{(1+\xi_n^2)^2}\Big)\times
\partial_{\xi_n}\sigma_{-3}((D_2^{*}D_1D_2^{*})^{-1})
(x_0)\bigg]\bigg|_{|\xi'|=1}d\xi_n\sigma(\xi')dx'\nonumber\\
&=&-8i h'(0)\times\frac{2 \pi i}{4!}\Big[\frac{-7i+26\xi_n+15i\xi_n^{2}}{(\xi_n+i)^3}
     \Big]^{(5)}|_{\xi_n=i}\Omega_4dx'+(9\pi i-4\pi)\Big(\lambda_1\langle dx_n,v\rangle-i\lambda_1\langle\xi',v\rangle \Big)\Omega_4dx'\nonumber\\
&=&\Big[\frac{55}{2}\pi h'(0)+(9\pi i-4\pi)\lambda_1\langle dx_n,v\rangle\Big]\Omega_4dx'.
\end{eqnarray}
Since $\Psi$ is the sum of the cases a), b) and c),
\begin{eqnarray*}
\Psi&=&\frac{(65-41i)\pi h'(0)}{8}\Omega_4dx'
+(18\pi\lambda_2+9\lambda_1\pi i)\rangle\langle dx_n,v\rangle\Omega_4dx'.
\end{eqnarray*}

\begin{thm}
Let $M$ be a $6$-dimensional
compact oriented manifold with the boundary $\partial M$ and the metric
$g^M$ as above, $D_i$ and $D_i^*~(i=1,2)$ be statistical de Rham Hodge Operators on $\widehat{M}$, then
\begin{eqnarray}
&&\widetilde{{\rm Wres}}[\pi^+D_1^{-1}\circ\pi^+(D_2^*D_1
      D_2^*)^{-1}]\nonumber\\
&=&128\pi^3\int_{M}
\bigg[-\frac{16}{3}s-8(\lambda^2_1+\lambda^2_2-8\lambda_1\lambda_2)|v|^2+\frac{1}{2}{\rm tr}[\lambda_2\nabla^{TM}_{e_{j}}(\varepsilon(v^*))c(e_{j})\nonumber\\
&&-\lambda_1c(e_{j})\nabla^{TM}_{e_{j}}(l(v))]
\bigg]d{\rm Vol_{M}}+\int_{\partial M}\big[\frac{(65-41i)\pi h'(0)}{8}\Omega_4dx'\nonumber\\
&&
+(18\pi\lambda_2+9\lambda_1\pi i)\rangle\langle dx_n,v\rangle\big]\Omega_4dx'.
\end{eqnarray}
where $s$ is the scalar curvature.
\end{thm}
\begin{cor}
When $\lambda_1=\lambda_2=1$, we get for a $4$-dimensional oriented
compact manifold $M$ with the boundary $\partial M$
\begin{eqnarray*}
&&\widetilde{{\rm Wres}}[\pi^+D^{-1}\circ\pi^+(D^*D
      D^*)^{-1}]\nonumber\\
&=&128\pi^3\int_{M}
\bigg[-\frac{16}{3}s+48|v|^2+\frac{1}{2}{\rm tr}[\nabla^{TM}_{e_{j}}(\varepsilon(v^*))c(e_{j})-c(e_{j})\nabla^{TM}_{e_{j}}(l(v))]
\bigg]d{\rm Vol_{M}}\nonumber\\
&&+\int_{\partial M}\Bigg(\frac{(65-41i)\pi h'(0)}{8}\Omega_4dx'
+(18\pi+9\pi i)\rangle\langle dx_n,v\rangle\Omega_4dx'\Bigg).
\end{eqnarray*}
where $s$ is the scalar curvature.
\end{cor}

\section*{Acknowledgements}
This work was supported by NSFC. 11771070 .
 The authors thank the referee for his (or her) careful reading and helpful comments.

\end{document}